%% file: pesin.tex
\documentclass[a4paper,10pt]{article}

\usepackage{geometry}
\usepackage[english]{babel}

\usepackage{amsmath}
\usepackage{amsfonts}
\usepackage{amsthm}
\usepackage{amssymb}
\usepackage{ifthen}
\usepackage{color}

\usepackage{enumitem}

\numberwithin{equation}{section}

\newtheorem{theorem}{Theorem}[section]
\newtheorem{lemma}[theorem]{Lemma}
\newtheorem{deflemma}[theorem]{Lemma and Definition}
\newtheorem{proposition}[theorem]{Proposition}
\newtheorem{corollary}[theorem]{Corollary}
\newtheorem{definition}[theorem]{Definition}

\newtheoremstyle{rem}{}{}{}{0pt}{\bfseries}{.}{5pt}{}
\theoremstyle{rem}
\newtheorem*{remark}{Remark}

\newtheoremstyle{ack}{}{}{}{\parindent}{\itshape\bfseries}{.}{5pt}{}
\theoremstyle{ack}

\newtheoremstyle{ass}{}{}{\itshape}{0pt}{\bfseries}{:}{5pt}{}
\theoremstyle{ass}
\newtheorem{ass}{Assumption}

\newcommand\eps\varepsilon
\newcommand\flow\varphi

\newcommand\abs[1]{\left\lvert#1\right\rvert}
\newcommand\norm[1]{\left\lVert#1\right\rVert}

\newcommand\E[1]{\mathbf{E}\!\left[#1\right]}

\newcommand\R{\mathbf{R}}
\newcommand\Z{\mathbf{Z}}
\newcommand\N{\mathbf{N}}
\newcommand\dx{\mathrm{d}}
\newcommand\sprod[1]{\left\langle#1\right\rangle}
\DeclareMathOperator*\esssup{ess\,\,sup}
\DeclareMathOperator{\Lipschitz}{Lip}
\newcommand\Lip{\Lipschitz}

\DeclareMathOperator{\id}{id}
\DeclareMathOperator{\graph}{graph}

\DeclareMathOperator{\dimension}{dim}
\DeclareMathOperator{\Emb}{Emb}

\DeclareMathOperator{\determinante}{det}

\newcommand\uTS[1]{H_{#1}(\omega,x)}
\newcommand\sTS[1]{E_{#1}(\omega,x)}
\newcommand\Lnorm[2][\empty]{\ifthenelse{\equal{#1}{\empty}} {\left\lVert#2\right\rVert_{(\omega,x),n}} {\left\lVert#2\right\rVert_{(\omega,x),#1}} }
\newcommand\LnormAt[3][\empty]{\ifthenelse{\equal{#1}{\empty}} {\left\lVert#3\right\rVert_{(\omega,#2),n}} {\left\lVert#3\right\rVert_{(\omega,#2),#1}} }
\newcommand\Linner[3][\empty]{\ifthenelse{\equal{#1}{\empty}} {\left\langle #2, #3 \right\rangle_{(\omega,x),n}} {\left\langle #2, #3 \right\rangle_{(\omega,x),#1}}}

\newcommand\sball[3][\empty]{\ifthenelse{\equal{#1}{\empty}} {B_{#3}^s\left(#2\right)} {B_{#3}^s\left(#1,#2\right)} }
\newcommand\uball[3][\empty]{\ifthenelse{\equal{#1}{\empty}} {B_{#3}^u\left(#2\right)} {B_{#3}^u\left(#1,#2\right)} }
\newcommand\uballat[4][\empty]{\ifthenelse{\equal{#1}{\empty}} {B_{#4,#3}^u\left(#2\right)} {B_{#4,#3}^u\left(#1,#2\right)} }
\newcommand\sLball[3][\empty]{\ifthenelse{\equal{#1}{\empty}} {\tilde B_{#3}^s\left(#2\right)} {\tilde B_{#3}^s\left(#1,#2\right)} }
\newcommand\uLball[3][\empty]{\ifthenelse{\equal{#1}{\empty}} {\tilde B_{#3}^u\left(#2\right)} {\tilde B_{#3}^u\left(#1,#2\right)} }
\newcommand\ball[3][\empty]{\ifthenelse{\equal{#1}{\empty}} {B_{#3}\left(#2\right)} {B_{#3}(#1,#2)} }
\newcommand\Lball[3][\empty]{\ifthenelse{\equal{#1}{\empty}} {\tilde B_{#3}\left(#2\right)} {\tilde B_{#3}\left(#1,#2\right)} }
\newcommand\pLball[3][\empty]{\ifthenelse{\equal{#1}{\empty}} {\tilde U \left(#2,#3\right)} {\tilde U_{#1} \left(#2,#3\right)} }

\newcommand\prodm[1]{\nu^{\N} \times \mu \left(#1\right)}
\newcommand\collLSM[2]{\mathcal{F}_{\Delta^l_{\omega}}(#1,#2)}
\newcommand\localLSM[2]{\tilde\Delta_{\omega}^{l}(#1,#2)}

\newcommand\LSM

\usepackage{url}

\usepackage[square,numbers]{natbib}

\newcommand\rref[1]{(\ref{#1})}

\title{Pesin's Formula for Random Dynamical Systems on $\R^d$}
\author{Moritz Biskamp\footnote{Institut f\"ur Mathematik, MA 7-4, Technische Universit\"at Berlin, Stra{\ss}e des 17.\ Juni 136, D-10623 Berlin, biskamp@math.tu-berlin.de}}

\begin{document}

\maketitle

\begin{abstract}
Pesin's formula relates the entropy of a dynamical system with its positive Lyapunov exponents. It is well known, that this formula holds true for random dynamical systems on a compact Riemannian manifold with invariant probability measure which is absolutely continuous with respect to the Lebesgue measure. We will show that this formula remains true for random dynamical systems on $\R^d$ which have an invariant probability measure absolutely continuous to the Lebesgue measure on $\R^d$. Finally we will show that a broad class of stochastic flows on $\R^{d}$ of a Kunita type satisfies Pesin's formula.
\end{abstract}

\vspace{1ex}

\vspace{2ex}

\input{introductionPesin}
\input{preliminaries}

\input{entropy}
\input{StbManifolds}

\input{derivativeestimates}
\input{absolute1}

\input{AbsContCondMeasures}

\input{Partition}

\input{below}
\input{ApplicationFlows}

\section*{Acknowledgement}
The present research was supported by the International Research Training Group {\it Stochastic Models of Complex Processes} funded by the German Research Council (DFG). The author gratefully thanks Michael Scheutzow and Simon Wasserroth from TU Berlin for their support and several fruitful discussions.

\bibliography{bibliography}
\bibliographystyle{amsplain}

\end{document}

%% file: introductionPesin.tex
\section{Introduction}

Entropy can be seen as a measure for uncertainty or for the chaotic behaviour of an evolution process. In information theory entropy is often interpreted as the minimal number of yes-no questions that are necessary to encrypt a finite signal. Here we are interested in the entropy of a dynamical system. A deterministic dynamical system preserving a smooth probability measure is the process generated by successive applications of a diffeomorphism on some space or manifold. The entropy for such a system given a partition of the space is roughly speaking the asymptotic exponential rate of yes-no questions necessary to encrypt the path of a particle evolving with this system with respect to this partition weighted with the invariant measure (see definition below).

Pesin's formula relates the entropy of a smooth dynamical system with its positive Lyapunov exponents. This remarkable formula was first established for deterministic dynamical systems on a compact Riemannian manifold preserving a smooth measure (see \citep{Pesin76}, \citep{Pesin77} and \citep{Pesin77b}). Pesin first proved general results concerning the existence of families of stable manifolds and their absolute continuity (see \citep{Pesin76}) and deduced therefrom the formula. Later, results were generalized to deterministic dynamical systems preserving only a Borel measure (see \citep{Ruelle79}, \citep{Fathi83}) and for dynamical systems with singularities (see \citep{Katok86}). In \citep{Barreira07} one finds a comprehensive and self-contained account on the theory dynamical systems with nonvanishing Lyapunov exponents, i.e. non-uniform hyperbolicity theory.

In this article we are interested in random dynamical systems, i.e. the evolution of the process generated by the successive application of {\it random} diffeomorphisms which will be assumed to be chosen independently according to some probability measure on the set of diffeomorphisms. Since it is much too restrictive to assume invariance of some probability measure for each diffeomorphism, the notion of invariance was extended to random dynamical system in \citep{Kifer86}: a probability measure is said to be invariant for a random dynamical system if the average over all possible diffeomorphisms preserves the measure (see definition below). 
The notion of entropy for random dynamical systems can not directly be deduced from the deterministic case, since in many interesting cases this quantity equals infinity (see \citep[Theorem II.1.2]{Kifer86}). Thus Kifer extended the notion of entropy in \citep{Kifer86} to random dynamical systems: Roughly speaking entropy of a random dynamical system given a partition of the state space is the asymptotic exponential rate of the {\it averaged} number of yes-no questions necessary to encrypt the path of a particle evolving with this system with respect to this partition weighted with the invariant measure. In terms of conditional entropy this coincides with the conditional entropy of the skew product given the randomness (see Section \ref{subsec:EntropyOfRandDiffeos}).
By this Pesin's results were generalized in \citep{Kifer86}, \citep{Ledrappier88} and \citep{Liu95} to random dynamical systems on compact Riemannian manifolds.

In this article we will extend the results to random dynamical systems on the non-compact space $\R^d$. The main application we have in mind when we consider random dynamical systems on $\R^d$ are stochastic flows on $\R^d$ with stationary and independent increments preserving a probability measure that is absolutely continuous to the Lebesgue measure on $\R^d$. In \citep{Arnold95} it was proven that under some regularity assumptions there is a one to one relation between random dynamical systems and stochastic flows of a Kunita type (see \citep{Kunita90}). In Section \ref{sec:ApplicationToFlows} we will show that the assumptions (see Section \ref{sec:rds}) are satisfied for a broad class of stochastic flows which have an invariant probability measure.

First we will introduce the formal concept of entropy and conditional entropy of partitions and measure preserving transformations (see Section \ref{sec:Preliminaries}). After we have defined random dynamical systems we will present some facts on entropy for random dynamical systems and the existence of Lyapunov exponents (see Section \ref{sec:EntropyOfRDS}).

To bound the entropy from below we have to construct a proper partition (see Section \ref{sec:ConstructionPartition}) such that the entropy of the random dynamical system given this partition can be bounded from below by its positive Lyapunov exponents. This partition will be constructed via local stable manifolds. Hence we will present the construction and the existence of local stable manifolds for random dynamical systems on $\R^d$ which have an invariant probability measure in Section \ref{sec:StableManifold}. This section follows very closely the general plan of \citep{Liu95}. Roughly speaking, the stable manifold at any point $x$ in space consists of those points that converge by application of the iterated functions with exponential speed to the iterated of $x$. One important construction within the proof is to define sets, nowadays called Pesin sets, which are chosen in such a way that one has uniform hyperbolicity on these sets (see Section \ref{sec:lyapunovmetric}), i.e. uniform bounds (in space {\it and} randomness) on the behaviour of the differential of the iterated maps (see Lemma \ref{lem:ExistenceOfl}).

In Sections \ref{sec:TheoremAbsoluteContinuity} and \ref{sec:AbsContCondMeasures} we state the theorems on the absolute continuity property. These basically say that the conditional measure with respect to the family of local stable manifolds of the volume on the state space is absolutely continuous (in fact, even equivalent) to the induced volume on the local stable manifolds. This is a crucial property within the construction of the partition mentioned in the previous paragraph and is proven in \citep{Biskamp11b}.

Finally in Section \ref{sec:proofPesin} we state the proof of Pesin's formula for random dynamical systems on $\R^d$ which have an invariant probability measure which is absolutely continuous to the Lebesgue measure on $\R^d$. First, we will bound the entropy from below following the proof of \citep[Chapter IV]{Liu95} and using the results from the previous sections (see Section \ref{sec:EstimateFromBelow}). The estimate from above (see Section \ref{sec:EstimateFromAbove}) was established in \citep{bargen10} for certain stochastic flows, but its proof can be applied to our situation by changing only two estimates in the proof.

Let us emphasize that we obviously can not equip the space of twice continuously differentiable diffeomorphisms on $\R^d$ with the uniform topology, as done in the case of a compact state space. Here we will use the topology induced by uniform convergence on compact sets (see \citep[Section 4.1]{Kunita90}). Clearly by this we lose the uniform bounds used in \citep{Liu95} to establish local stable manifolds (in particular the counterpart of Lemma \ref{lem:ExistenceOfr}). To replace these uniform bounds we need to assume certain integrability assumptions (see Section \ref{sec:rds}). As already mentioned we will show in Section \ref{sec:ApplicationToFlows} that all these assumptions are satisfied for a broad class of stochastic flows on $\R^d$.

%% file: preliminaries.tex
\section{Preliminaries} \label{sec:Preliminaries}

We will give a short introduction into (conditional) entropy of partitions and measure preserving transformations, mainly following \citep{Liu95}.

\subsection{Measurable Partitions}
Let $(X,\mathcal{B},\mu)$ a Lebesgue space. A partition of $X$ is a collection of non-empty disjoint sets that cover $X$. Subsets of $X$ that are unions of elements of a partition $\xi$ are called $\xi$-sets.

A countable family $\{B_\alpha : \alpha \in \mathcal{A}\}$ of measurable $\xi$-sets is said to be a basis of the partition $\xi$ if  for any two elements $C$ and $C'$ of $\xi$ there exists an $\alpha \in \mathcal{A}$ such that either $C \subset B_\alpha$, $C' \not\subset B_\alpha$ or $C' \subset B_\alpha$, $C \not\subset B_\alpha$. A partition which has a basis is called a measurable partition.

For $x \in X$ we will denote by $\xi(x)$ the element of the partition $\xi$ that contains $x$. If $\xi, \xi'$ are measurable partitions of $X$, we will write $\xi \leq \xi'$ if $\xi'(x) \subset \xi(x)$ for $\mu$-almost every $x \in X$.

For any system of measurable partitions $\{\xi_\alpha\}$ of $X$ there exists a product $\bigvee_\alpha \xi_\alpha$ defined as the measurable partition $\xi$ that satisfies the following two properties: 1) $\xi_\alpha \leq \xi$ for all $\alpha$; 2) if $\xi_\alpha \leq \xi'$ for all $\alpha$ then $\xi \leq \xi'$. Furthermore for any measurable partition $\{\xi_\alpha\}$ of $X$ there exists an intersection $\bigwedge_\alpha \xi_\alpha$ defined as the measurable partition $\xi$ that satisfies the following two properties: 1) $\xi_\alpha \geq \xi$ for all $\alpha$; 2) if $\xi_\alpha \geq \xi'$ for all $\alpha$ then $\xi \geq \xi'$.

Let us introduce the factor space $X / \xi$ of $X$ with respect to a partition $\xi$ whose points are the elements of $\xi$. Its measurable structure and measure $\mu_\xi$ is defined as follows: Let $p$ be the map that maps $x \in X$ to $\xi(x)$, then a set $Z$ is considered to be measurable if $p^{-1}(Z) \in \mathcal{B}$ and we define $\mu_\xi(Z) := \mu(p^{-1}(Z))$. Let us remark that if $\xi$ is a measurable partition then $X/\xi$ is again a Lebesgue space.

For measurable partitions $\xi_n$, $n \in \N$ and $\xi$ of $X$ the symbol $\xi_n \nearrow \xi$ indicates that $\xi_1 \leq \xi_2 \leq \dots$ and $\bigvee_n \xi_n = \xi$. Similarly the symbol $\xi_n \searrow \xi$ indicates that $\xi_1\geq \xi_2 \geq \dots$ and $\bigwedge_n \xi_n = \xi$.

For a measurable partition $\xi$ the $\sigma$-algebra generated by $\xi$ consists of those measurable sets of $X$ that are (arbitrary) unions of $\xi$-sets. Conversely for any sub-$\sigma$-algebra there exists a generating measurable partition. Thus in the future we will often not distinguish between the $\sigma$-algebra and its generating partition.

One very important property of measurable partitions of a Lebesgue space is that associated to such a partition $\xi$ there exists a unique system of measures $\{\mu_C\}_{C \in \xi}$ satisfying the following two conditions:
\begin{enumerate}
\item $(C,\mathcal{B}|_C,\mu_C)$ is a Lebesgue space for $\mu_\xi$-a.e. $C \in X/\xi$
\item for every $A \in \mathcal{B}$ the map $C \mapsto \mu_C(A \cap C)$ is measurable on $X/\xi$ and
\begin{align*}
\mu(A) = \int_{X/\xi} \mu_C(A \cap C) \dx \mu_\xi(C).
\end{align*}
\end{enumerate}
Such a system of measures $\{\mu_C\}_{C\in \xi}$ is called a {\it canonical system of conditional measures} of $\mu$ associated to the partition $\xi$. 

More detailed informations on measurable partitions can be found in \citep[Section 0.2]{Liu95}.

\subsection{Conditional Entropies of Measurable Partitions}

Let us again assume that $(X,\mathcal{B}, \mu)$ is a Lebesgue space. If $\xi$ is a measurable partition of $X$ and $C_1, C_2, \dots$ are the elements of $\xi$ with positive $\mu$ measure then we define the {\it entropy} of the partition $\xi$ by
\begin{align*}
 H_\mu(\xi) =
\begin{cases}
- \sum_k \mu(C_k)\log(\mu(C_k)) & \text{if } \mu(X \setminus \bigcup_k C_k) = 0 \\
+ \infty & \text{if } \mu(X \setminus \bigcup_k C_k) > 0.
\end{cases}
\end{align*}
Let us remark that the sum in the first part can be finite or infinite.

If $\xi$ and $\eta$ are two measurable partitions of $X$, then almost every partition $\xi_B$, which is the restriction $\xi|_B$ of $\xi$ to $B \in X/\eta$, has a well defined entropy $H_{\mu_B}(\xi_B)$. This is a non-negative measurable function on the factor space $X/\eta$, called the {\it conditional entropy} of $\xi$ with respect to $\eta$. Let us set
\begin{align*}
H_\mu(\xi|\eta) := \int_{X/\eta} H_{\mu_B}(\xi_B) \dx \mu_\eta(B),
\end{align*}
which is the {\it mean conditional entropy} of $\xi$ with respect to $\eta$. This number can be finite or infinite. If $\eta$ is the trivial partition whose single element is $X$ itself, then clearly $H_\mu(\xi|\eta)$ coincides with $H_\mu(\xi)$. Furthermore it is easy to see that
\begin{align} \label{eq:alternativeCondEntropy}
H_\mu(\xi|\eta) = - \int_X \log\left(\mu_{\eta(x)}(\xi(x) \cap \eta(x))\right) \dx \mu(x).
\end{align}
If the partition $\eta$ generates the $\sigma$-algebra $\mathcal{G}$ then the conditional entropy can be expressed in terms of conditional probabilities, i.e.
\begin{align*}
H_\mu(\xi|\eta) =H_\mu(\xi|\mathcal{G}) :=  - \int_X \sum_{C \in \xi} \mu(C | \mathcal{G}) \log \mu(C | \mathcal{G}) \dx \mu.
\end{align*}
Let us state some basic properties of the conditional entropy (see \citep[Section 0.3]{Liu95}).
\begin{lemma}
Let $\xi_n$, $\eta_n$ for $ n \in \N$ and $\xi$, $\eta$ and $\zeta$ be measurable partitions of $X$. Then we have
\begin{enumerate}
\item if $\xi_n \nearrow \xi$ then $H_\mu(\xi_n|\eta) \nearrow H_\mu(\xi|\eta)$;
\item if $\xi_n \searrow \xi$ and $\eta$ satisfies $H_\mu(\xi_1|\eta) < \infty$ then $H_\mu(\xi_n|\eta) \searrow H_\mu(\xi|\eta)$;
\item $H_\mu(\xi \vee \eta|\zeta) = H_\mu(\xi|\zeta) + H_\mu(\eta|\xi \vee \zeta)$;
\item if $\eta_n \nearrow \eta$ and $\xi$ satisfies $H_\mu(\xi|\eta_1) < \infty$ then $H_\mu(\xi|\eta_n) \searrow H_\mu(\xi|\eta)$;
\item if $\eta_n \searrow \eta$ then $H_\mu(\xi|\eta_n) \nearrow H_\mu(\xi|\eta)$.
\end{enumerate}
Further if $(X_i, \mathcal{B}_i, \mu_i)$ for $i=1,2$ are two Lebesgue spaces and $T$ is a measure-preserving transformation from $(X_1, \mathcal{B}_1, \mu_1)$ to $(X_2, \mathcal{B}_2, \mu_2)$, then for any measurable partition $\xi$ and $\eta$ of $X_2$ we have
\begin{align*}
H_{\mu_1}(T^{-1}\xi| T^{-1}\eta) = H_{\mu_2}(\xi|\eta).
\end{align*}
\end{lemma}

\begin{proof}
For the proof of property i) - v) see \citep{Rohlin67} and for the last one see \citep[Section 0.3]{Liu95}.
\end{proof}

\subsection{Conditional Entropies of Measure-Preserving Transformations}

Let us consider a  measure preserving transformation $T: X \to X$ and a $\sigma$-algebra $\mathcal{A} \subset \mathcal{B}$ with $T^{-1} \mathcal{A} \subset \mathcal{A}$ and denote the generating partition of $\mathcal{A}$ by $\zeta_{0}$. Then we can define the entropy of the transformation $T$ in the sense of Kifer (see \citep{Kifer86}) as follows.

\begin{deflemma} \label{deflemma:Entropy}
For any measurable partition $\xi$ with $H_\mu(\xi | \mathcal{A}) < +\infty$ the following limit exists
\begin{align*}
h^{\mathcal{A}}_\mu (T, \xi) = \lim_{n \to +\infty} \frac{1}{n} H_\mu \left( \bigvee_{i=0}^{n-1} T^{-i} \xi \bigg| \zeta_{0}\right).
\end{align*}
The number $h^{\mathcal{A}}_\mu (T, \xi)$ is called the $\mathcal{A}$-conditional entropy of $T$ with respect to $\xi$. Furthermore
\begin{align*}
h^\mathcal{A}_\mu(T) := \sup_{\xi} h^\mathcal{A}_\mu(T,\xi) \quad \text{and} \quad h_\mu(T) := \sup_{\xi} h^{\{\emptyset,X\}}_\mu(T,\xi)
\end{align*}
are called the $\mathcal{A}$-entropy of $T$ and entropy of $T$ respectively. Here the supremum is either taken over all partition $\xi$ with finite entropy or over all finite partitions.
\end{deflemma}

\begin{proof}
See \citep{Kifer86} and \citep[Section 0.4 and Section 0.5]{Liu95}.
\end{proof}

If we want to define entropy for any measurable partition of $X$ we need to assume that the $\sigma$-algebra $\mathcal{A}$ is invariant under the transformation $T$, i.e. the following definition.

\begin{definition} \label{def:Entropy2}
Assume that $T^{-1}\mathcal{A} = \mathcal{A}$. Then for any measurable partition $\xi$ of $X$ we define
\begin{align*}
h^{\mathcal{A}}_{\mu}(T,\xi) = H_{\mu}\left(\xi\bigg|\bigvee_{k=1}^{+\infty}T^{-k}\xi \vee \zeta_{0}\right).
\end{align*}
\end{definition}

\begin{remark}
For any measurable partition $\xi$ that satisfies $H_\mu(\xi | \mathcal{A}) < +\infty$ the Definition \ref{deflemma:Entropy} and \ref{def:Entropy2} coincide (see \citep[Remark 0.5.1]{Liu95}).
\end{remark}

%% file: entropy.tex
\section{Entropy and Lyapunov Exponents of Random Dynamical Systems} \label{sec:EntropyOfRDS}

In this section we will first introduce the notion of random dynamical systems. Then we will define its entropy and state the multiplicative ergodic theorem to define Lyapunov exponents. Here we are following \citep[Chapter I]{Liu95}.

\subsection{Random Dynamical Systems} \label{sec:rds}

Let us abbreviate the set of two-times differentiable diffeomorphisms on $\R^d$ by $\Omega$. The topology on $\Omega$ is the one induced by uniform convergence on compact sets for all derivatives up to order 2 as described in \citep[Section 3.1]{Kunita90}. With this topology $\Omega$ becomes a separable Banach space. Let us fix a Borel probability measure $\nu$ on $(\Omega, \mathcal{B}(\Omega))$, where $\mathcal{B}(\Omega)$ denotes the Borel $\sigma$-algebra of $\Omega$.

We are interested in ergodic theory of the evolution process generated by successive applications of randomly chosen maps from $\Omega$. These maps will be assumed to be independent and identically distributed with law $\nu$. Thus let
\begin{align*}
\left(\Omega^\N, \mathcal{B}(\Omega)^\N, \nu^\N\right) = \prod_{i =0}^{+\infty} (\Omega, \mathcal{B}(\Omega),\nu)
\end{align*}
be the infinite product of copies of the measure space $(\Omega, \mathcal{B}(\Omega),\nu)$. Let us define for every $\omega = (f_0(\omega), f_1(\omega), \dots) \in \Omega^\N$ and $n \geq 0$
\begin{align*}
f^0_\omega = \id, \qquad f^n_\omega = f_{n-1}(\omega) \circ f_{n-2}(\omega) \circ \dots \circ f_0(\omega).
\end{align*}
The random dynamical system generated by these composed maps, e.g. $\{f^n_\omega : n \geq 0, \omega \in (\Omega^{\N}, \mathcal{B}(\Omega)^{\N},\nu^{\N})\}$ will be referred to as $\mathcal{X}^+(\R^d, \nu)$.

Let us further define the two important spaces $\Omega^\N \times \R^d$ and $\Omega^\Z \times \R^d$, both equipped with the  product $\sigma$-algebras  $\mathcal{B}(\Omega)^\N \times \mathcal{B}(\R^d)$ and $\mathcal{B}(\Omega)^\Z \times \mathcal{B}(\R^d)$ respectively. As already mentioned above $\Omega$ is a separable Banach space by the choice of the uniform topology on compact sets. Hence we have
\begin{align*}
\mathcal{B}(\Omega)^\N \times \mathcal{B}(\R^d) &= \mathcal{B}(\Omega^\N \times \R^d), \\
\mathcal{B}(\Omega)^\Z \times \mathcal{B}(\R^d) &= \mathcal{B}(\Omega^\Z \times \R^d).
\end{align*}
Further let us denote by $\tau$ the left shift operator on $\Omega^\N$ and $\Omega^\Z$, namely
\begin{align*}
 f_n(\tau \omega) = f_{n+1}(\omega)
\end{align*}
for all $\omega = (f_0(\omega), f_1(\omega), \dots) \in \Omega^\N$, $n \geq 0$ and $\omega = (\dots, f_{-1}(\omega), f_0(\omega), f_1(\omega), \dots) \in \Omega^\Z$, $n \in \Z$ respectively. Finally let
\begin{align*}
F&: \Omega^\N \times \R^d \to \Omega^\N \times \R^d, &(\omega,x) \mapsto (\tau \omega, f_0(\omega) x),\\
G&: \Omega^\Z \times \R^d \to \Omega^\Z \times \R^d, &(\omega,x) \mapsto (\tau \omega, f_0(\omega) x).
\end{align*}
The function $F$ is often called the {\it skew product} of the system. The two systems $(\Omega^\N \times \R^d, F)$ and $(\Omega^\Z \times \R^d, G)$ will allow us to see the random dynamical system somehow as a deterministic one on the product space.

\begin{definition} \label{def:invariantMeasure}
A Borel probability measure $\mu$ on $\R^d$ is called an invariant measure of $\mathcal{X}^+(\R^d, \nu)$ if
\begin{align*}
\int_\Omega \mu(f^{-1}(\cdot)) \dx \nu(f) = \mu.
\end{align*}
\end{definition}

From now let us assume that there exists an invariant measure $\mu$ of $\mathcal{X}^+(\R^d, \nu)$ and let us denote the random dynamical system associated with $\mu$ by $\mathcal{X}^+(\R^d,\nu,\mu)$. From \citep[Lemma I.2.3]{Kifer86} we have the following Lemma, which relates the notion of invariance defined above with the invariance with respect to the skew product, i.e. the function $F$ on $\Omega^\N \times \R^d$.

\begin{lemma}
Let $\mu$ be a probability measure on $\R^d$. Then $\mu$ is an invariant measure of $\mathcal{X}^+(\R^d, \nu)$ (in the sense of Definition \ref{def:invariantMeasure}) if and only if $\nu^\N \times \mu$ is $F$-invariant, i.e. $(\nu^\N\times\mu) \circ F^{-1} = \nu^\N\times\mu$.
\end{lemma}

\begin{proof}
See \citep[Lemma I.2.3]{Kifer86}.
\end{proof}

Let us denote the tangent space at some point $y \in \R^{d}$ by $T_{y}\R^{d}$. Although this is quite unusual for systems on $\R^{d}$ we will stick to the notation from \citep{Liu95}. Let us define the following map, in differential geometry known as the exponential function, for $y \in \R^d$
\begin{align*}
\exp_y : \R^d \cong T_y \R^d \to \R^d, \quad x \mapsto \exp_y(x) := x + y,
\end{align*}
where $\cong$ means that the two spaces are isometrically isomorphic and thus can be identified. In the following we will use this often implicitely. Then we can define for $(\omega,x) \in \Omega^\N \times \R^d$ and $n \geq 0$ the map
\begin{align*}
F_{(\omega,x),n} : T_{f^{n}_\omega x} \R^d \to T_{f^{n+1}_\omega x}\R^d; \qquad
F_{(\omega,x),n} := \exp_{f^{n+1}_\omega x}^{-1} \circ f_n(\omega) \circ \exp_{f^{n}_\omega x},
\end{align*}
which is the evolution process centered around the trajectory of $x$, i.e. $F_{(\omega,x),n}(0) = 0$ for all $n \geq 0$. Throughout this article we will assume that the random dynamical system $\mathcal{X}^+(\R^d, \nu, \mu)$ satisfies the following integrability assumptions on $\nu$ and $\mu$:

\begin{ass} \label{ass1}
Let $\nu$ and $\mu$ satisfy
\begin{align*}
\log^+\abs{D_x f_0(\omega)} &\in \mathcal{L}^1(\nu^\N\times\mu),
\end{align*}
where $\abs{D_x f_0(\omega)}$ denotes the operator norm of the differential as a linear operator from $T_x \R^d$ to $T_{f_0(\omega) x} \R^d$ induced by the Euclidean scalar product and $\log^+(a) = \max\{\log(a); 0\}$.
\end{ass}

\begin{ass}\label{ass2}
Let $\nu$ and $\mu$ satisfy
\begin{align*}
\log\left(\sup_{\xi \in B_x(0,1)} \abs{D^2_\xi F_{(\omega,x),0}}\right) &\in \mathcal{L}^1(\nu^\N\times\mu),\\
\log\left(\sup_{\xi \in B_x(0,1)}\abs{D^2_{F_{(\omega,x),0}(\xi)} F^{-1}_{(\omega,x),0}}\right) &\in \mathcal{L}^1(\nu^\N\times\mu),
\end{align*}
where $B_x(0,r)$ denotes the open ball in $T_x \R^d$ around the origin with radius $r >0$ and $D^2$ is the second derivative operator.
\end{ass}

\begin{ass} \label{ass1b}
Let $\nu$ and $\mu$ satisfy
\begin{align*}
\log \abs{D_0F_{(\omega,x),0}^{-1}} = \log\abs{D_{f_0(\omega)x} f_0(\omega)^{-1}} &\in \mathcal{L}^1(\nu^\N\times\mu).
\end{align*}
\end{ass}

\begin{ass} \label{ass2b}
Let $\nu$ and $\mu$ satisfy
\begin{align*}
\log\abs{\determinante D_x f_0(\omega)} \in \mathcal{L}^1(\nu^\N \times \mu).
\end{align*}
\end{ass}

\begin{ass} \label{ass3}
Let $\mu$ and $\nu$ satisfy for all $n \in \N$
\begin{align*}
\sup_{\xi \in B(x,1)} \log^+\abs{D_\xi f^{n}_{\omega}} \in \mathcal{L}^{1}\left(\nu^\N \times \mu\right).
\end{align*}
\end{ass}

Assumption \ref{ass1} is necessary for the application of the multiplicative ergodic theorem (see next section), whereas Assumption \ref{ass1b} is used in Lemma \ref{lem:DerivativeEstimate} to achieve an estimate on the derivative of the inverse. Assumption \ref{ass2} is used in Lemma \ref{lem:ExistenceOfr} to get an uniform bound on the Lipschitz constant of the derivative and its inverse on some specific set $\Gamma_0 \subset \Omega^\N \times \R^d$. Finally we need Assumption \ref{ass2b} in Setion \ref{sec:EstimateFromBelow} for the final proof to bound the entropy from below and Assumption \ref{ass3} for the proof of the estimation from above (see Setion \ref{sec:EstimateFromAbove}).

Let us remark that Assumption \ref{ass2} can be relaxed by taking not the unit ball in $T_x\R^d$ into consideration but some ball with positive radius. Furthermore obviously Assumption \ref{ass3} implies Assumption \ref{ass1}, but we want to make clear which integrability assumption is used at what point of the proof.

\subsection{Measure-Theoretic Entropies of Random Diffeomorphisms} \label{subsec:EntropyOfRandDiffeos}

In this section we will define the notion of entropy for random dynamical systems. We are closely following \citep{Kifer86} and \citep{Liu95}.

\begin{deflemma}
For any finite partition $\xi$ of $\R^d$ the limit
\begin{align*}
h_\mu(\mathcal{X}^+(\R^d, \nu), \xi) := \lim_{n \to +\infty} \frac{1}{n} \int_{\Omega^\N} H_\mu\left(\bigvee_{k=0}^{n-1} (f^k_\omega)^{-1} \xi \right) \dx \nu^\N(\omega)
\end{align*}
exists. The number $h_\mu(\mathcal{X}^+(\R^d, \nu), \xi)$ is called the entropy of $\mathcal{X}^+(\R^d,\nu,\mu)$ with respect to $\xi$. The number
\begin{align*}
h_\mu(\mathcal{X}^+(\R^d, \nu)) := \sup_{\xi} h_\mu(\mathcal{X}^+(\R^d, \nu), \xi)
\end{align*}
is called the entropy of $h_\mu(\mathcal{X}^+(\R^d, \nu), \xi)$. Here the supremum is either taken over all partition $\xi$ with finite entropy or over all finite partitions.
\end{deflemma}

Let us denote the projection from $\Omega^\Z \times \R^d$ to $\Omega^\N \times \R^d$ by $P$, i.e.
\begin{align*}
P: \Omega^\Z \times \R^d \to \Omega^\N \times \R^d, \quad (\omega,x) \mapsto (\omega^+, x),
\end{align*}
where $\omega^+ := (f_0(\omega), f_1(\omega), \dots)$ for $\omega \in \Omega^\Z$ and let us define the following $\sigma$-algebras
\begin{align*}
\sigma_0 &:= \left\{\Gamma \times \R^d : \Gamma \in \mathcal{B}(\Omega^\N) \right\};\\
\sigma^+ &:= \left\{\prod_{-\infty}^{-1} \Omega \times \Gamma \times \R^d : \Gamma \in \mathcal{B}\left(\prod_0^{+\infty} \Omega\right) \right\};\\
\sigma &:= \left\{\Gamma' \times \R^d : \Gamma' \in \mathcal{B}(\Omega^\Z) \right\}.
\end{align*}
Clearly these $\sigma$-algebras correspond to the measurable partitions $\{\{\omega\}\times \R^d : \omega \in \Omega^\N\}$ of $\Omega^\N \times \R^d$, $\{\prod_{-\infty}^{-1}\Omega \times \{\omega\}\times \R^d : \omega \in \prod_0^{+\infty} \Omega \}$ of $\Omega^\Z \times \R^d$ and $\{\{\omega\}\times \R^d : \omega \in \Omega^\Z\}$ of $\Omega^\Z \times \R^d$ respectively. We will often use the same symbols for both, the $\sigma$-algebra and the partition.

Then we have the following result from \citep{Liu95}, which relates the entropy of $\mathcal{X}^+(\R^d,\nu, \mu)$ in the sense of the previous definition to the conditional entropy defined in Definition \ref{deflemma:Entropy}.

\begin{theorem} \label{thm:EqualityOfEntropies1}
If $\xi = \{A_1, \dots, A_n\}$ is a finite partition of $\R^d$ and $\eta = \{B_1,\dots,B_m\}$ a finite partition of $\Omega^\N$ then we have
\begin{align*}
h_\mu(\mathcal{X}^+(\R^d,\nu), \xi) = h^{\sigma_0}_{\nu^\N \times \mu}(F, \xi \times \eta),
\end{align*}
where $\xi \times \eta := \{A_i\times B_j: 1 \leq i \leq n, 1 \leq j \leq m\}$. Furthermore this yields
\begin{align*}
h_\mu(\mathcal{X}^+(\R^d,\nu)) = h^{\sigma_0}_{\nu^\N \times \mu}(F).
\end{align*}
\end{theorem}

\begin{proof}
See \citep[Theorem I.2.2]{Liu95}.
\end{proof}

The following proposition from \citep{Liu95} justifies to transfer the invariant measure from $\Omega^{\N}\times\R^{d}$ to $\Omega^{\Z}\times\R^{d}$.

\begin{proposition} \label{prop:ExistenceOfMuStar}
For every invariant probability measure $\mu$ of $\mathcal{X}^+(\R^d,\nu)$ there exists a unique Borel probability measure $\mu^*$ on $\Omega^\Z \times \R^d$ such that $G\mu^* = \mu^*$ and $P\mu^* = \nu^\N \times \mu$.
\end{proposition}

\begin{proof}
See \citep[Proposition I.1.2]{Liu95}.
\end{proof}

The following theorem from \citep{Liu95} relates the entropy of $G$ on $\Omega^{\Z}\times\R^{d}$ with the entropy of $F$ on $\Omega^{\N}\times\R^{d}$. It will be useful for the proof of the estimation of the entropy from below (see Section\ref{sec:EstimateFromBelow}).

\begin{theorem} \label{thm:EqualityOfEntropies2}
For $\mathcal{X}^+(\R^d,\nu,\mu)$ it holds that
\begin{align*}
h_{\nu^\N \times \mu}^{\sigma_0}(F) = h_{\mu^*}^{\sigma^+}(G) = h_{\mu^*}^\sigma(G).
\end{align*}
\end{theorem}

\begin{proof}
See \citep[Theorem I.2.3]{Liu95}.
\end{proof}

\subsection{Multiplicative Ergodic Theorem and Lyapunov Exponents}

By Assumption \ref{ass1} in the previous section the multiplicative ergodic theorem yields the existence of linear subspaces with corresponding Lyapunov exponents, which play an extraordinary important role in the analysis of dynamical systems. The following theorem is \citep[Theorem I.3.2]{Liu95}.

\begin{theorem} \label{thm:met}
For the given system $\mathcal{X}^+(\R^d,\nu,\mu)$ there exists a Borel set $\Lambda_0 \subset \Omega^\N \times \R^d$ with $\nu^\N \times \mu (\Lambda_0) = 1$, $F \Lambda_0 \subset \Lambda_0$ such that:
\begin{enumerate}
\item For every $(\omega,x) \in \Lambda_0$ there exists a sequence of linear subspaces of $T_x\R^d$
\begin{align*}
\{0\} = V^{(0)}_{(\omega,x)} \subset V^{(1)}_{(\omega,x)} \subset \ldots \subset V^{(r(x))}_{(\omega,x)}
\end{align*}
and numbers (called Lyapunov exponents)
\begin{align*}
\lambda^{(1)}(x) < \lambda^{(2)}(x) < \ldots < \lambda^{(r(x))}(x)
\end{align*}
($\lambda^{(1)}(x)$ may be $-\infty$), which depend only on $x$, such that
\begin{align*}
\lim_{n\to +\infty} \frac{1}{n} \log \abs{D_x f^n_\omega \xi} = \lambda^{(i)}(x)
\end{align*}
for all $\xi \in V^{(i)}_{(\omega,x)} \!\setminus\! V^{(i-1)}_{(\omega,x)}$, $1 \leq i \leq r(x)$, and in addition
\begin{align*}
\lim_{n \to +\infty} \frac{1}{n} \log \abs{D_x f^n_\omega} &= \lambda^{(r(x))}(x)\\
\lim_{n \to +\infty} \frac{1}{n} \log \abs{\determinante (D_x f^n_\omega)} &= \sum_i \lambda^{(i)}(x) m_i(x)
\end{align*}
where $m_i(x) = \dimension(V^{(i)}_{(\omega,x)}) - \dimension(V^{(i-1)}_{(\omega,x)})$, which depends only on $x$ as well. Moreover, $r(x), \lambda^{(i)}(x)$ and $V^{(i)}_{(\omega,x)}$ depend measurably on $(\omega,x) \in \Lambda_0$ and
\begin{align*}
r(f_0(\omega)x) = r(x),\quad \lambda^{(i)}(f_0(\omega)x) = \lambda^{(i)}(x), \quad 
D_x f_0(\omega) V^{(i)}_{(\omega,x)} = V^{(i)}_{F(\omega,x)},
\end{align*}
for each $(\omega,x) \in \Lambda_0$, $1 \leq i \leq r(x)$.

\item For each $(\omega,x) \in \Lambda_0$, we introduce
\begin{align}
 \rho^{(1)}(x) \leq  \rho^{(2)}(x) \leq \ldots \leq \rho^{(d)}(x)
\end{align}
to denote $\lambda^{(1)}(x), \dots, \lambda^{(1)}(x), \dots, \lambda^{(i)}(x), \dots, \lambda^{(i)}(x), \dots \lambda^{(r(x))}(x), \dots, \lambda^{(r(x))}(x)$ with $\lambda^{(i)}(x)$ being repeated $m_i(x)$ times. Now, for $(\omega,x) \in \Lambda_0$, if $\{\xi_1, \dots, \xi_d\}$ is a basis of $T_x \R^d$ which satisfies
\begin{align*}
\lim_{n \to +\infty} \frac{1}{n} \log \abs{D_x f^n_\omega \xi_i} = \rho^{(i)}(x)
\end{align*}
for every $1 \leq i \leq d$, then for every two non-empty disjoint subsets $P,Q \subset \{1, \dots, d\}$ we have
\begin{align*}
\lim_{n \to +\infty} \frac{1}{n} \log \gamma (D_x f^n_\omega E_P, D_x f^n_\omega E_Q) = 0,
\end{align*}
where $E_P$ and $E_Q$ denote the subspaces of $T_x \R^d$ spanned by the vectors $\{\xi_i\}_{i \in P}$ and $\{\xi_j\}_{j \in Q}$ respectively and $\gamma(\cdot, \cdot)$ denotes the angle between the two associated subspaces.
\end{enumerate}
\end{theorem}

For more details on the multiplicative ergodic theorem for random dynamical systems and Lyapunov exponents see for example \citep{Arnold98} or \citep[Section I.3]{Liu95}. The angle between to linear subspaces $E$ and $E'$ of a tangent space $T_x\R^d$ for some $x \in \R^d$ is defined by
\begin{align*}
\gamma(E,E') := \inf\left\{\cos^{-1}\left(\sprod{\xi,\xi'}\right) : \xi \in E, \xi' \in E', \abs{\xi} = \abs{\xi'} = 1\right\},
\end{align*}
where $\sprod{\cdot,\cdot}$ denotes the Euclidean scalar product on $T_x\R^d$.

\subsection{Pesin Formula}

Now we are able to formulate the main theorem of this article

\begin{theorem} \label{thm:Pesin}
Let $\mathcal{X}(\R^d, \nu)$ be a random dynamical system which has an invariant measure $\mu$ and satisfying Assumptions \ref{ass1} - \ref{ass3}. Further assume that the invariant measure $\mu$ is absolutely continuous with respect to the Lebesgue measure on $\R^d$ then we have
\begin{align*}
h_\mu(\mathcal{X}(\R^d, \nu)) = \int \sum_i \lambda^{(i)}(x)^+ m_i(x) \dx \mu(x).
\end{align*}
\end{theorem}

\begin{proof}
The proof of the theorem can be found in Section \ref{sec:proofPesin}. For the proof we need some preparation, which will be done in the following sections.
\end{proof}

%% file: StbManifolds.tex
\section{Local and Global Stable Manifolds} \label{sec:StableManifold}

In this section we will mainly follow the book of Liu and Qian \citep[Chapter III]{Liu95}. In general proofs are only given, if there is a need to change arguments due to the non-compactness of $\R^d$ as the state space of the random dynamical system. Otherwise we will state the reference for the proof.

\subsection{Lyapunov Metric and Pesin Sets} \label{sec:lyapunovmetric}

Let us define for some interval $[a,b]$, $a < b \leq 0$, of the real line the set
\begin{align*}
\Lambda_{a,b} := \left\{ (\omega,x) \in \Lambda_0 : \lambda^i(x) \notin [a,b] \text{ for all } i\in 1,\dots,r(x) \right\},
\end{align*}
where $\Lambda_{0}$ was defined in of Theorem \ref{thm:met}. Because of $F\Lambda_0 \subset \Lambda_0$ and the invariance of the Lyapunov exponents we have $F\Lambda_{a,b} \subset \Lambda_{a,b}$. For $(\omega,x) \in \Lambda_{a,b}$ and $n\geq 1$ define the following linear subspaces of $T_{x}\R^{d}$ and $T_{f^{n}_{\omega}x}\R^{d}$ respectively by
\begin{align*}
\sTS{0} &:= \bigcup_{\lambda^{(i)}(x) < a} V^{(i)}_{(\omega,x)},\qquad & \uTS{0} &:= \sTS{0}^\bot,\\
\sTS{n} &:= D_xf^n_\omega \sTS{0}, \qquad & \uTS{n} &:= D_xf^n_\omega \uTS{0}.
\end{align*}
For $n,l \geq 1$ let us denote the iterated functions by
\begin{align*}
f^0_n(\omega) := \id, \qquad f^l_n(\omega) = f_{n+l-1}(\omega) \circ \dots \circ f_n(\omega).
\end{align*}
and we will denote the derivative of $f^l_n(\omega)$ at $f^n_\omega x$ by $T^l_n(\omega,x) := D_{f^n_\omega x}f^l_n(\omega)$ and its restriction to $\sTS{n}$ and $\uTS{n}$ respectively by
\begin{align*}
S^l_n(\omega,x) := T^l_n(\omega,x)|_{\sTS{n}}, \qquad U^l_n(\omega,x) := T^l_n(\omega,x)|_{\uTS{n}}.
\end{align*}

Let us now fix $k\geq 1$ and $0 < \eps \leq \min\{1,(b-a)/(200d)\}$ and let us assume that the set
\begin{align*}
\Lambda_{a,b,k} := \{(\omega,x) \in \Lambda_{a,b} : \dimension \sTS{0} = k\}
\end{align*}
is non-empty. Then we have the following lemma from \citep[Lemma III.1.1]{Liu95}.

\begin{lemma} \label{lem:ExistenceOfl}
There exists a measurable function $l: \Lambda_{a,b,k} \times \N \to (0,+\infty)$ such that for each $(\omega,x) \in \Lambda_{a,b,k}$ and $n,l \geq 1$ we have
\begin{enumerate}
\item $\abs{S^l_n(\omega,x)\xi} \leq l(\omega,x,n) e^{(a+\eps)l}\abs{\xi}$, for all $\xi \in \sTS{n}$;
\item $\abs{U^l_n(\omega,x)\eta} \geq l(\omega,x,n)^{-1} e^{(b-\eps)l}\abs{\eta}$, for all $\eta \in \uTS{n}$;
\item $\gamma(\sTS{n+l},\uTS{n+l}) \geq l(\omega,x,n)^{-1}e^{-\eps l}$;
\item $l(\omega,x,n+l) \leq l(\omega,x,n) e^{\eps l}$,
\end{enumerate}
where $\gamma(\cdot,\cdot)$ again denotes the angle between two linear subspaces.
\end{lemma}

\begin{proof}
See \citep[Proof of Lemma III.1.1]{Liu95}.
\end{proof}

Let us fix a number $l' \geq 1$ such that the set
\begin{align*}
\Lambda^{l'}_{a,b,k,\eps}:= \left\{(\omega,x) \in \Lambda_{a,b,k} : l(\omega,x,0) \leq l'\right\}
\end{align*}
is non-empty. These sets where we have uniform bounds on the derivative by Lemma \ref{lem:ExistenceOfl} are often called {\it Pesin sets}. Since on these sets the function $l$ is uniformly bounded by definition we can show continuity of the subspaces $\sTS{0}$ and $\uTS{0}$ there, which is \citep[Lemma III.1.2]{Liu95}.

\begin{lemma} \label{lem:ContinuousDependeceOfStableSpaces}
The linear subspaces $\sTS{0}$ and $\uTS{0}$ depend continuously on $(\omega,x) \in \Lambda^{l'}_{a,b,k,\eps}$.
\end{lemma}

\begin{proof}
Although this is \citep[Lemma III.1.2]{Liu95} we will say a few words concerning the topology on $\Omega^\N$. As mentioned in Section \ref{sec:rds} the topology on $\Omega$ will be the one induced by uniform convergence on compact sets for all derivatives up to order 2 (see \citep[Chapter 4]{Kunita90}). Thus on $\Omega^\N$ we will use the usual topology of uniform convergence on finitely many elements. The space of all $k$-dimensional subspaces of $T_x\R^d \cong \R^d$ will be equipped with the Grasmannian metric, by which this space is compact.

Let $(\omega_n,x_n) \in \Lambda^{l'}_{a,b,k,\eps}$ be a sequence converging to $(\omega,x) \in \Lambda^{l'}_{a,b,k,\eps}$. By compactness of the Grassmanian there exists a subsequence of $\{(\omega_n,x_n)\}_n$ (denoted by the same symbols) such that $E_0(\omega_n,x_n)$ converges to some linear subspace $E$. Clearly $E$ is a subspace of $T_x\R^d$. For each $\zeta \in E$ there is a sequence $\xi_n \in E_0(\omega_n,x_n)$ such that $\abs{\zeta -\xi_n} \to 0$. Because for $n \in \N$ we have by Lemma \ref{lem:ExistenceOfl} that
\begin{align*}
\abs{T^l_0(\omega_n,x_n)\xi_n} = \abs{S^l_0(\omega_n,x_n)\xi_n} \leq l' e^{(a+\eps)l}\abs{\xi_n} \to l' e^{(a+\eps)l}\abs{\zeta}
\end{align*}
we only need to show that the left hand side converges to $\abs{T^l_0(\omega,x)\zeta}$. Since $\{\xi_n\}_{n\in\N} \cup \{\zeta\}$ is a compact set in $\R^d$ and the derivatives of each component of $\omega_n$ converge uniformly on compact sets we get for all $\zeta \in E$
\begin{align*}
\abs{T^l_0(\omega,x)\zeta} \leq l' e^{(a+\eps)l}\abs{\zeta}.
\end{align*}
Then Lemma \ref{lem:ExistenceOfl} implies that actually $\zeta \in E(\omega,x)$, which completes the proof.
\end{proof}

For $(\omega,x) \in \Lambda^{l'}_{a,b,k,\eps}$ and $n \in \N$ Lemma \ref{lem:ExistenceOfl} also allows us to define an inner product $\Linner{~}{~}$ on $T_{f^n_\omega x}\R^d$ such that
\begin{align*}
\Linner{\xi}{\xi'} &= \sum_{l=0}^{+\infty} e^{-2(a + 2\eps)l} \Big\langle S^l_n(\omega,x)\xi, S^l_n(\omega,x)\xi'\Big\rangle, && \text{for } \xi,\xi' \in \sTS{n}\\
\Linner{\eta}{\eta'} &= \sum_{l=0}^{n} e^{2(b - 2\eps)l} \left\langle \left[U^l_{n-l}(\omega,x)\right]^{-1}\eta, \left[U^l_{n-l}(\omega,x)\right]^{-1}\eta'\right\rangle, && \text{for } \eta,\eta' \in \uTS{n}.\\
\end{align*}
and $\sTS{n}$ and $\uTS{n}$ are orthogonal with respect to $\Linner{~}{~}$. Thus we can define the norms
\begin{align*}
\Lnorm{\xi} &:= \left[\Linner{\xi}{\xi}\right]^{\frac{1}{2}}& &\text{for } \xi \in \sTS{n};\\
\Lnorm{\eta} &:= \left[\Linner{\eta}{\eta}\right]^{\frac{1}{2}}& &\text{for } \eta \in \uTS{n};\\
\Lnorm{\zeta} &:= \max\left\{\Lnorm{\xi} ,\Lnorm{\eta}\right\}& &\text{for } \zeta = \xi + \eta \in \sTS{n} \oplus \uTS{n}.
\end{align*}

The sequence of norms $\{\Lnorm{\cdot}\}_{n\in \N}$ is usually called {\it Lyapunov metric} at the point $(\omega,x)$. By the definition of the inner product and by Lemma \ref{lem:ContinuousDependeceOfStableSpaces} the inner product $\Linner{~}{~}$ depends continuously on $(\omega,x) \in \Lambda^{l'}_{a,b,k,\eps}$. Now we can state \citep[Lemma III.1.3]{Liu95}.

\begin{lemma} \label{lem:EstimatesOnLnorm}
Let $(\omega,x) \in \Lambda^{l'}_{a,b,k,\eps}$. Then the Lyapunov metric at $(\omega,x)$ satisfies for each $n\in \N$
\begin{enumerate}
\item $\Lnorm[n+1]{S^1_{n}(\omega,x)\xi} \leq e^{a+2\eps} \Lnorm{\xi}\quad$ for $\xi \in \sTS{n}$;
\item $\Lnorm[n+1]{U^1_{n}(\omega,x)\eta} \geq e^{b-2\eps} \Lnorm{\eta}\quad$ for $\eta \in \uTS{n}$;
\item $\frac{1}{2} \abs{\zeta} \leq \Lnorm{\zeta} \leq Ae^{2\eps n}\abs{\zeta}\quad$ for $\zeta \in T_{f^n_\omega x}\R^d$, where $A = 4(l')^2(1-e^{-2\eps})^{-\frac{1}{2}}$.
\end{enumerate}
\end{lemma}

\begin{proof}
See \citep[Lemma III.1.3]{Liu95}.
\end{proof}

To the end of this section we will prove the following important lemma. The proof is similar to the one of \citep[Lemma III.1.4]{Liu95} but has to be adapted to the situation of a non-compact state space. We will use $\Lip(\cdot)$ to denote the Lipschitz constant of a function with respect to the Euclidean norm $\abs{\cdot}$ if not mentioned otherwise.

\begin{lemma} \label{lem:ExistenceOfr}
 There exists a Borel set $\Gamma_0 \subset \Omega^\N \times \R^d$ and a measurable function $r: \Gamma_0 \to (0, \infty)$ such that $\nu^\N \times \mu (\Gamma_0) = 1$, $F\Gamma_0 \subset \Gamma_0$ and for all $(\omega, x) \in \Gamma_0$ 
\begin{enumerate}
 \item the map 
\begin{align*}
 F_{(\omega,x),0} = \exp_{f_0(\omega)x}^{-1} \circ f_0(\omega) \circ \exp_x : T_x\R^d \ni B_x(0,1) \to T_{f_0(\omega)x}\R^d,
\end{align*}
where $B_x(0,1)$ denotes the unit ball in $T_x\R^d$ around $0$, satisfies
\begin{align*}
 \Lip(D_{\cdot}F_{(\omega,x),0}) &\leq r(\omega,x),\\
 \Lip(D_{F_{(\omega,x),0}(\cdot)}F_{(\omega,x),0}^{-1}) &\leq r(\omega,x);
\end{align*}
\item $r(F^n(\omega,x)) = r(\tau^n\omega,f^n_\omega x) \leq r(\omega, x) e^{\eps n}$.
\end{enumerate}
\end{lemma}

\begin{proof}
Let us define the function $r': \Omega^\N \times \R^d$ by
\begin{align*}
r'(\omega,x) := &\max\left\{ \sup_{\xi \in B_x(0,1)} \abs{D^2_\xi F_{(\omega,x),0}}; \sup_{\xi \in B_{x}(0,1)}\abs{D^2_{F_{(\omega,x),0}(\xi)} F^{-1}_{(\omega,x),0}} \right\},
\end{align*}
where $D^2$ is the second derivative operator. Then by Assumption \ref{ass2} we have $\log(r') \in \mathcal{L}^1(\nu^{\N}\times\mu)$. According to Birkhoff's ergodic theorem there exists a measurable set $\Gamma_0 \subseteq \Omega^\N \times \R^d$ with $\nu^{\N}\times\mu(\Gamma_0) = 1$ and $F\Gamma_0 \subseteq \Gamma_0$ such that for all $(\omega,x) \in \Gamma_0$ we have
\begin{align*}
\lim_{n\to \infty} \frac{1}{n} \log\left(r'(F^n(\omega,x))\right) = 0.
\end{align*}
Thus it follows that
\begin{align*}
r(\omega,x) := \sup_{n \geq 0}\left\{r'(F^n(\omega,x)) e^{-\eps n} \right\}
\end{align*}
is finite at each point $(\omega,x) \in \Gamma_0$ and $r$ satisfies the requirements of the lemma by the mean value theorem.
\end{proof}

\subsection{Local Stable Manifolds}

Fix a number $r' \geq 1$ such that the Borel set
\begin{align*}
\Lambda^{l',r'}_{a,b,k,\eps} := \left\{(\omega,x) \in \Lambda^{l'}_{a,b,k,\eps} \cap \Gamma_0 : r(\omega,x)\leq r'\right\}
\end{align*}
is non-empty. For ease of notation we will abbreviate $\Lambda' :=  \Lambda^{l',r'}_{a,b,k,\eps}$. Then we can introduce the notion of {\it local stable manifolds} as in \citep[Section III.3]{Liu95}.

\begin{definition}
Let $X$ be a metric space and let $\{D_x\}_{x\in X}$ be a collection of subsets of $\R^d$. We call $\{D_x\}_{x \in X}$ a continuous family of $C^1$ embedded $k$-dimensional discs in $\R^d$ if there is a finite open cover $\{U_i\}_{i=1,\dots,l}$ of $X$ such that for each $U_i$ there exists a continuous map $\theta_i: U_i \to \Emb^1(B^k,\R^d)$ such that $\theta_i(x)B^k = D_x$, $x \in U_i$, where $B^k := \{\xi \in \R^k: \abs{\xi} < 1\}$ is the open unit ball in $\R^k$ and the topology on $\Emb^1(B^k,\R^d)$ is the one induced by uniform convergence on compact sets.
\end{definition}

Then we have the main theorem of this section, which states the existence of local stable manifolds and its representation (see \citep[Theorem III.3.1]{Liu95}).

\begin{theorem} \label{thm:localStableManifold}
For each $n \in \N$ there exists a continuous family of $C^1$ embedded $k$-di\-men\-sio\-nal discs $\{W_n(\omega,x)\}_{(\omega,x)\in \Lambda'}$ in $\R^d$  and there exist numbers $\alpha_n, \beta_n$ and $\gamma_n$ which depend only on $a, b, k, \eps, l'$ and $r'$ such that the following hold true for every $(\omega,x) \in \Lambda'$:
\begin{enumerate}
\item There exists a $C^{1,1}$ map
  \begin{align*}
  h_{(\omega,x),n} : O_n(\omega,x) \to H_n(\omega,x),
  \end{align*}
  where $O_n(\omega,x)$  is an open subset of $E_n(\omega,x)$ which contains $\{\xi \in E_n(\omega,x) : \abs{\xi} \leq \alpha_n\}$, such that
\begin{enumerate}
  \item $h_{(\omega,x),n} (0) = 0$;
  \item $\Lip(h_{(\omega,x),n}) \leq \beta_n$, $\Lip(D_\cdot h_{(\omega,x),n}) \leq \beta_n$;
  \item $W_n(\omega,x) = \exp_{f^n_\omega x} \graph(h_{(\omega,x),n})$ and $W_n(\omega,x)$ is tangent to $E_n(\omega,x)$ at the point $f^n_\omega x$;
\end{enumerate}
\item $f_n(\omega) W_n(\omega,x) \subseteq W_{n+1}(\omega,x)$
\item $d^s(f^l_n(\omega)y,f^l_n(\omega)z) \leq \gamma_n e^{(a+4\eps)l} d^s(y,z)$ for $y,z \in W_n(\omega,x)$, $l\in \N$, where $d^s(\cdot, \cdot)$ is the distance along $W_m(\omega,x)$ for $m \in \N$;
\item $\alpha_{n+1} = \alpha_n e^{-5\eps}, \beta_{n+1} = \beta_n e^{7\eps}$ and $\gamma_{n+1} = \gamma_n e^{2\eps}$.
\end{enumerate}
\end{theorem}

\begin{proof}
For the proof see \citep[Theorem III.3.1]{Liu95}. But let us emphasize that the following estimates are essential for the proof and that they are satisfied in our situation. Put
\begin{align*} 
 \eps_0 := e^{a + 4 \eps} - e^{a + 2\eps}, \quad c_0 :=4 A r' e^{2\eps}, \quad r_0 := c_0^{-1}\eps_0.
\end{align*}
Then one can easily check by using the results from Section \ref{sec:lyapunovmetric} that for $l \geq 0$ the map
\begin{align*}
F_{(\omega,x),l} = \exp^{-1}_{f^{l+1}_\omega x} \circ f_l(\omega) \circ \exp_{f^l_\omega x} : \left\{\xi \in T_{f^l_\omega x}\R^d : \Lnorm[l]{\xi} \leq r_0 e^{-3\eps l}\right\} \to T_{f^{l+1}_\omega x}\R^d
\end{align*}
satisfies
\begin{align*} 
\Lip_{\norm{\cdot}}(D_\cdot F_{(\omega,x),l}) \leq c_0 e^{3\eps l} \qquad \text{and} \qquad
\Lip_{\norm{\cdot}}(F_{(\omega,x),l} - D_0 F_{(\omega,x),l}) \leq \eps_0,
\end{align*}
where $\Lip_{\norm{\cdot}}$ denotes the Lipschitz constant with respect to $\Lnorm[l]{\cdot}$ and $\Lnorm[l+1]{\cdot}$. Furthermore if we define for $n, l \geq 0$ the composition by
\begin{align*}
F^0_n(\omega,x) = \id,\qquad F^l_n(\omega,x) := F_{(\omega,x),n+l-1} \circ \dots \circ F_{(\omega,x),n}
\end{align*}
then for $(\xi_0,\eta_0) \in \exp^{-1}_x(W_0(\omega,x))$ with $\Lnorm[0]{(\xi_0,\eta_0)} \leq r_0$ we get for every $n \geq 0$ the estimate 
\begin{align*} 
\Lnorm[n]{F^n_0(\omega,x)(\xi_0,\eta_0)} \leq \Lnorm[0]{(\xi_0,\eta_0)} e^{(a + 6\eps)n}.
\end{align*}
\end{proof}

\subsection{Global Stable Manifolds} \label{sec:globalStableMfld}

This section deals with the existence of {\it global stable manifolds}, which are constructed using local stable manifolds. Denote
\begin{align} \label{eq:hatLambda0}
\hat\Lambda_0 := \Lambda_0 \cap \Gamma_0, \qquad \hat \Lambda_{a,b,k} := \Lambda_{a,b,k} \cap \hat\Lambda_0,
\end{align}
where $\Lambda_0$ comes from Theorem \ref{thm:met} and $\Gamma_0$ from Lemma \ref{lem:ExistenceOfr}. Let $\{l'_m\}_{m \in \N}$ and $\{r'_m\}_{m \in\N}$ be a monotone sequence of positive numbers such that $l'_m \nearrow +\infty$ and $r'_m \nearrow +\infty$ as $m \to +\infty$. Then we have for all $m \in \N$
\begin{align*}
\Lambda_{a,b,k,\eps}^{l'_m,r'_m} \subset \Lambda_{a,b,k,\eps}^{l'_{m+1},r'_{m+1}}
\end{align*}
and
\begin{align*}
 \hat\Lambda_{a,b,k} = \bigcup_{m=1}^{+\infty}\Lambda_{a,b,k,\eps}^{l'_m,r'_m}.
\end{align*}
If we denote
\begin{align*}
\{[a_n,b_n]\}_{n\in\N} := \left\{[a,b]: a < b \leq 0,\text{ $a$ and $b$ are rational}\right\}
\end{align*}
and let
\begin{align*}
\eps_n := \frac{1}{2}\min\left\{1,\frac{1}{(200d)}(b_n - a_n)\right\},
\end{align*}
then we have
\begin{align*}
 \hat\Lambda_0 = \left\{\bigcup_{n=1}^{+\infty}\bigcup_{k=1}^{d}\bigcup_{m=1}^{+\infty}\Lambda_{a_n,b_n,k,\eps_n}^{l'_m,r'_m}\right\} \cup \left\{(\omega,x) \in \hat\Lambda_0: \lambda^{(i)}(x) \geq 0, 1 \leq i \leq r(x)\right\}.
\end{align*}
The following theorem, which is \citep[Theorem III.3.2]{Liu95}, then states the existence of global stable manifolds.

\begin{theorem} \label{thm:globalStableManifold}
Let $(\omega,x) \in \hat \Lambda_0\!\setminus\!\left\{(\omega,x) \in \hat\Lambda_0: \lambda^{(i)}(x) \geq 0, 1 \leq i \leq r(x)\right\}$ and let $\lambda^{(1)}(x) < \cdots < \lambda^{(p)}(x)$ be the strictly negative Lyapunov exponents at $(\omega,x)$. Define $W^{s,1}(\omega,x) \subset \cdots \subset W^{s,p}(\omega,x)$ by
\begin{align*}
W^{s,i}(\omega,x) := \left\{y \in \R^d : \limsup_{n\to\infty} \frac{1}{n} \log \abs{f^n_\omega x - f^n_\omega y} \leq \lambda^{(i)}(x)\right\}
\end{align*}
for $1 \leq i \leq p$. Then $W^{s,i}(\omega,x)$ is the image of $V^{(i)}_{(\omega,x)}$ under an injective immersion of class $C^{1,1}$ and is tangent to $V^{(i)}_{(\omega,x)}$ at $x$. In addition, if $y \in W^{s,i}(\omega,x)$ then
\begin{align*}
\limsup_{n\to\infty} \frac{1}{n} \log d^s(f^n_\omega x, f^n_\omega y) \leq \lambda^{(i)}(x)
\end{align*}
where $d^s(~,~)$ denotes the distance along the submanifold $f^n_\omega W^{s,i}(\omega,x)$.
\end{theorem}

\begin{proof}
See \citep[Theorem III.3.2]{Liu95}.
\end{proof}

\begin{definition}
For $(\omega,x)\in \Omega^\N\times\R^d$ the global stable manifold $W^s(\omega,x)$ is defined by
\begin{align*}
W^s(\omega,x) := \left\{y\in\R^d : \limsup_{n\to\infty}\frac{1}{n} \log\abs{f^n_\omega x - f^n_\omega y} < 0 \right\}.
\end{align*}
\end{definition}

Let $\Lambda' = \Lambda_{a,b,k,\eps}^{l',r'}$ be as considered before Theorem \ref{thm:localStableManifold}. For $(\omega,x) \in \Lambda'$ let $\lambda^{(1)}(x) < \dots < \lambda^{(i)}(x)$ be the Lyapunov exponents smaller than $a$. Then one can see that
\begin{align*} 
 W^{s,i}(\omega,x) = \left\{y\in\R^d : \limsup_{n\to\infty}\frac{1}{n} \log \abs{f^n_\omega x - f^n_\omega y} \leq a \right\}.
\end{align*}
Thus if $(\omega,x) \in \hat \Lambda_0\!\setminus\!\left\{(\omega,x) \in \hat\Lambda_0: \lambda^{(i)}(x) \geq 0, 1 \leq i \leq r(x)\right\}$ and $\lambda^{(1)}(x) < \cdots < \lambda^{(p)}(x)$ are the strictly negative Lyapunov exponents at $(\omega,x)$ then we get
\begin{align*}
 W^s(\omega,x) = W^{s,p}(\omega,x)
\end{align*}
and hence $W^s(\omega,x)$ is the image of $V^{(p)}_{(\omega,x)}$ under an injective immersion of class $C^{1,1}$ and is tangent to $V^{(p)}_{(\omega,x)}$ at $x$.

%% file: derivativeestimates.tex
\subsection{Another Estimate on the Derivative}

For the proof of the absolute continuity theorem (see \citep{Biskamp11b}), which will be stated in the next section, we need the following estimate on the derivative.

\begin{lemma} \label{lem:DerivativeEstimate}
 There exists a set $\Gamma_1 \subset \Omega^\N \times \R^d$, with $F\Gamma_1 \subset \Gamma_1$ and $\nu^\N \times \mu(\Gamma_1) = 1$ such that for every $\delta \in (0,1)$, there exists a positive measurable function $C_\delta$ defined on $\Gamma_1$ such that for every $(\omega,x) \in \Gamma_1$ and $n \geq 0$ one has
\begin{align*}
\abs{D_0F^{-1}_{(\omega,x),n}} \leq C_\delta(\omega,x) e^{\delta n}.
\end{align*}
\end{lemma}

\begin{proof}
By Assumption \ref{ass1b} we have $\log\abs{D_0F^{-1}_{(\omega,x),0}} \in \mathcal{L}^1(\nu^\N\times\mu)$ and hence we get by Birkhoff's ergodic theorem the existence of a measurable set $\Gamma_1 \subset \Omega^\N \times \R^d$, which satisfies $F\Gamma_1 \subset \Gamma_1$ and $\nu^\N \times \mu(\Gamma_1) = 1$ such that for all $(\omega,x) \in \Gamma_1$
\begin{align*}
\frac{1}{n} \log \abs{D_0F^{-1}_{(\omega,x),n}} = \frac{1}{n} \log \abs{D_0F^{-1}_{F^n(\omega,x),0}} \to 0.
\end{align*}
Thus for all $\delta \in (0,1)$ we find a measurable function $C_\delta$ such that for all $n \geq 0$ and $(\omega,x) \in \Omega^{\N}\times\R^{d}$
\begin{align*}
\abs{D_0F^{-1}_{(\omega,x),n}} \leq C_\delta(\omega,x) e^{\delta n}.
\end{align*}
\end{proof}

Let us fix some $C' \geq 1$ such that the set
\begin{align*}
\Lambda_{a,b,k,\eps}^{l',r',C'} := \left\{(\omega,x) \in \Lambda_{a,b,k,\eps}^{l',r'} \cap \Gamma_1: C_\eps(\omega,x) \leq C' \right\}
\end{align*}
is non-empty and let us abbreviate in the following
\begin{align*}
\Delta := \Lambda_{a,b,k,\eps}^{r',l',C'}.
\end{align*}
The parameters for the definition of $\Delta$ will be fixed from now on.

%% file: absolute1.tex
\section{Absolute Continuity Theorem} \label{sec:TheoremAbsoluteContinuity}

In this section we will state the absolute continuity theorem. To do so we will need some preparation.

Let us choose a sequence of approximating compact sets $\{\Delta^l\}_{l}$ with $\Delta^l \subset \Delta$ and $\Delta^l \subset \Delta^{l+1}$ such that $\prodm{\Delta\!\setminus\!\Delta^l} \to 0$ for $l \to \infty$ and let us fix arbitrarily such a set $\Delta^l$. For $(\omega,x) \in \Delta$ and $r >0$ define
\begin{align*}
\pLball[\Delta,\omega]{x}{r} := \exp_x \left(\left\{\zeta \in T_x\R^d : \Lnorm[0]{\zeta} < r\right\}\right)
\end{align*}
and for $(\omega,x) \in \Delta^l$ let
\begin{align*}
V_{\Delta^l}((\omega,x),r) := \left\{(\omega',x') \in \Delta^l : d(\omega,\omega') < r, x' \in \pLball[\Delta,\omega]{x}{r}\right\},
\end{align*}
where the distance $d$ in $\Omega^\N$ is as before the one induced by uniform convergence on compact sets for all derivatives up to order $2$.
Let us denote the collection of local stable manifolds $\{W_0(\omega,x)\}_{(\omega,x)\in \Delta^l}$ which was constructed in Theorem \ref{thm:localStableManifold} in the following by $\{W_{loc}(\omega,x)\}_{(\omega,x)\in \Delta^l}$. Since by Theorem \ref{thm:localStableManifold} this is a continuous family of $C^1$ embedded $k$-dimensional discs and $\Delta^l$ is compact there exists uniformly on $\Delta^l$ a number $\delta_{\Delta^l} > 0$ such that for any $0 < q \leq \delta_{\Delta^l}$ and $(\omega',x')\in V_{\Delta^l}((\omega,x),q/2)$ the local stable manifold $W_{loc}(\omega',x')$ can be represented in local coordinates with respect to $(\omega,x)$, i.e. there exists a $C^1$ map
\begin{align*}
\phi : \left\{ \xi \in \sTS{0} : \Lnorm[0]{\xi} < q \right\} \to \uTS{0}
\end{align*}
with
\begin{align*}
\exp_x^{-1}\left(W_{loc}(\omega',x') \cap \pLball[\Delta,\omega]{x}{q}\right) = \graph(\phi).
\end{align*}
By choosing $\delta_{\Delta^l}$ even smaller we can ensure, that for all $0 < q \leq \delta_{\Delta^l}$, $(\omega,x) \in \Delta^{l}$ and $(\omega',x')\in V_{\Delta^l}((\omega,x),q/2)$
\begin{align*}
\sup \left\{\Lnorm[0]{D_\xi \phi} : \xi \in \sTS{0}, \Lnorm[0]{\xi} < q  \right\} \leq\frac{1}{3}.
\end{align*}

Let us fix until the end of the section $(\omega,x) \in \Delta^l$ and $0 < q \leq \delta_{\Delta^l}$. Then we denote by $\Delta^l_\omega := \left\{x \in \R^d: (\omega,x) \in \Delta^l\right\}$ the $\omega$-section of $\Delta^{l}$ and by $\collLSM{x}{q}$ the collection of local stable submanifolds $W_{loc}(\omega,y)$ passing through $y \in \Delta^l_{\omega} \cap \pLball[\Delta,\omega]{x}{q/2}$ and set
\begin{align*}
&\localLSM{x}{q} := \bigcup_{y \in \Delta^l_{\omega} \cap \pLball[\Delta,\omega]{x}{q/2}} W_{loc}(\omega,y) \cap \pLball[\Delta,\omega]{x}{q}.
\end{align*}

Let us introduce the notion of transversal manifolds to the collection of local stable manifolds $\collLSM{x}{q}$.

\begin{definition} \label{def:TransversalMfld}
A submanifold $W$ of $\R^d$ is called transversal to the family $\collLSM{x}{q}$ if the following hold true
\begin{enumerate}
\item $W \subset \pLball[\Delta,\omega]{x}{q}$ and $\exp_{x}^{-1}W$ is the graph of a $C^1$ map 
\begin{align*}
\psi: \left\{\eta \in \uTS{0} : \Lnorm[0]{\eta} < q\right\} \to \sTS{0};
\end{align*}
\item W intersects any $W_{loc}(\omega,y)$, $y \in \Delta^l_{\omega} \cap \pLball[\Delta,\omega]{x}{q/2}$, at exactly one point and this intersection is transversal, i.e. $T_{z}W \oplus T_{z}W_{loc}(\omega,y) = \R^d$ where $z = W \cap W_{loc}(\omega, y)$.
\end{enumerate}
\end{definition}

For a submanifold $W$ of $\R^d$ transversal to $\collLSM{x}{q}$ let
\begin{align*}
\norm{W} := \sup_\eta\Lnorm[0]{\psi(\eta)} + \sup_{\eta}\Lnorm[0]{D_\eta \psi}
\end{align*}
where the supremum is taken over $\{\eta \in \uTS{0}: \Lnorm[0]{\eta} < q\}$ and $\psi$ is the map representing $W$ as in Definition \ref{def:TransversalMfld}.

Consider two submanifolds $W^1$ and $W^2$ transversal to $\collLSM{x}{q}$. By the choice of $\delta_{\Delta^l}$ each local stable manifold passing through $y \in \Delta^l_\omega \cap \pLball[\Delta,\omega]{x}{q/2}$ can be represented via some function $\phi$, whose norm of the derivative with respect to the Lyapunov metric is bounded by $1/3$. Thus the following map, which is usually called {\it Poincar\'e map} or {\it holonomy map}, is well defined. Define
\begin{align*}
P_{W^1,W^2}: W^1 \cap \localLSM{x}{q} \to W^2 \cap \localLSM{x}{q}
\end{align*}
by
\begin{align*}
P_{W^1,W^2} : z = W^1\cap W_{loc}(\omega,y) \mapsto W^2\cap W_{loc}(\omega,y),
\end{align*}
for each $y \in \Delta^l_{\omega} \cap \pLball[\Delta,\omega]{x}{q/2}$. Since the collection of local stable manifolds is by Theorem \ref{thm:localStableManifold} a continuous family of $C^1$ embedded $k$-dimensional discs $P_{W^1,W^2}$ is a homeomorphism.Denoting the Lebesgue measures on $W^i$ by $\lambda_{W^i}$ for $i = 1,2$. we can define absolute continuity of the family $\collLSM{x}{q}$.

\begin{definition}
The family $\collLSM{x}{q}$ is said to be absolutely continuous if  there exists a number $\eps_{\Delta^l_{\omega}}(x,q) > 0$ such that for any two submanifolds $W^1$ and $W^2$ transversal to $\collLSM{x}{q}$ and satisfying $\norm{W^{i}} \leq \eps_{\Delta^l_{\omega}}(x,q)$, $i = 1,2$, the Poincar\'e map $P_{W^1,W^2}$ constructed as above is absolutely continuous with respect to $\lambda_{W^1}$ and $\lambda_{W^2}$, i.e. $\lambda_{W^1} \approx \lambda_{W^2}\circ P_{W^1,W^2}$.
\end{definition}

Then we have the following main theorem, often called absolute continuity theorem, which is proved for this case in \citep{Biskamp11b}. Let us denote the Lebesgue measure on $\R^d$ by $\lambda$.

\begin{theorem} \label{thm:ACT}
Let $\Delta^l$ be given as above. There exist numbers $0 < q_{\Delta^l} < \delta_{\Delta^l}/2$ and $\eps_{\Delta^l} > 0$ such that for every $(\omega,x) \in \Delta^l$ and $0 < q \leq q_{\Delta^l}$:
\begin{enumerate}
\item The family $\collLSM{x}{q}$ is absolutely continuous.
\item If $\lambda(\Delta_{\omega}^l) > 0$ and $x$ is a density point of $\Delta_{\omega}^l$ with respect to $\lambda$, then for every two submanifolds $W^1$ and $W^2$ transversal to $\collLSM{x}{q_{\Delta^l}}$ and satisfying $\norm{W^{i}} \leq \eps_{\Delta^l}$, $i = 1,2$, any Poincar\'e map $P_{W^1,W^2}$ is absolutely continuous and the Jacobian $J(P_{W^1,W^2})$ satisfies the inequality
\begin{align*}
\frac{1}{2} \leq J(P_{W^1,W^2})(y) \leq 2
\end{align*}
for $\lambda_{W^1}$-almost all $y \in W^1 \cap \localLSM{x}{q_{\Delta^l}}$.
\end{enumerate}
\end{theorem}

\begin{proof}
See \citep{Biskamp11b}.
\end{proof}

%% file: AbsContCondMeasures.tex
\section{Absolute Continuity of Conditional Measures} \label{sec:AbsContCondMeasures}

In this section we will state the main conclusion of the absolute continuity theorem namely Theorem \ref{thm:AbsContCondMeasures}, which  roughly speaking says that the conditional measure with respect to the family of local stable manifolds of the volume on the state space is absolutely continuous (in fact, even equivalent) to the induced volume on the local stable manifolds.

Let us start with the following proposition, which is \citep[Proposition 6.1]{Liu95}.

\begin{proposition} \label{prop:AbsCondMeasures}
Let $(X,\mathcal{B},\nu)$ be a Lebesgue space and let $\alpha$ be a measurable partition of $X$. If $\hat \nu$ is another probability measure on $\mathcal{B}$ which is absolutely continuous with respect to $\nu$, then for $\hat \nu$-almost all $x \in X$ the conditional measure $\hat \nu_{\alpha(x)}$ is absolutely continuous with respect to $\nu_{\alpha(x)}$ and
\begin{align*}
\frac{\dx \hat \nu_{\alpha(x)}}{\dx \nu_{\alpha(x)}} = \frac{g|_{\alpha(x)}}{\int_{\alpha(x)} g \dx \nu_{\alpha(x)}}
\end{align*}
where $g = \dx \hat \nu / \dx \nu$.
\end{proposition}

\begin{proof}
See \citep[Proposition 6.1]{Liu95}.
\end{proof}

Let $\Delta^l$ be a compact set as in the previous Section. Without loss of generality we can and will assume that $q_{\Delta^l} = \eps_{\Delta^l}$. Let us fix a point $(\omega,x) \in \Delta^l$ until the end of this section such that $\lambda(\Delta^l_\omega) >0$ and $x$ is a density point of $\Delta^l_\omega$ with respect to $\lambda$. Let us introduce the following abbreviations
\begin{align*}
\hat U &:= \tilde U_{\Delta,\omega}(x, q_{\Delta^l})\\
\hat B^1 &:= \left\{ \xi \in \sTS{0} : \LnormAt[0]{x}{\xi} <  q_{\Delta^l}\right\} \\
\hat B^2 &:= \left\{ \eta \in \uTS{0} : \LnormAt[0]{x}{\eta} <  q_{\Delta^l}\right\}.
\end{align*}
We will denote by $\beta$ the measurable partition $\left\{\exp_x\left(\{\xi\}\times \hat B^2 \right)\right\}_{\xi \in \hat B^1}$ of $\hat U$ and by $\alpha$ the partition of $\localLSM{x}{q_{\Delta^l}}$ into local stable manifolds, i.e. $\left\{ W_{loc}(\omega,y) \cap \hat U \right\}_{y \in \Delta^l_\omega \cap \tilde U_{\Delta,\omega}\left(x,q_\Delta^l/2\right)}$. Since $\left\{ W_{loc}(\omega,y)\right\}_{y\in\Delta^l_\omega}$ is a continuous family of $C^1$ $k$-dimensional embedded discs $\alpha$ is a measurable partition of $\localLSM{x}{q_{\Delta^l}}$. Further we define the sets
\begin{align*}
I &:= \beta(x) \cap \localLSM{x}{q_{\Delta^l}} \\
\intertext{and for $N \subset I$}
[N] &:= \bigcup_{z \in N} \alpha(z).
\end{align*}
Since $q_{\Delta^l}$ is chosen such that each local stable manifold $W_{loc}(\omega,y)$ for $y \in \Delta^l_\omega \cap \tilde U_{\Delta,\omega}\left(x,q_\Delta^l/2\right)$ can be expressed as a function on $\sTS{0}$ we have $[I] = \localLSM{x}{q_{\Delta^l}}$. Because $x$ is a density point of $\Delta^l_\omega$ with respect to $\lambda$ we have $\lambda(\Delta^l_\omega \cap \tilde U_{\Delta,\omega}(x,q_{\Delta^l}/2)) >0$ which implies that $\lambda(\localLSM{x}{q_{\Delta^l}}) = \lambda([I]) > 0$.

The restriction of $\beta$ to $[I]$ will be denoted by $\beta_I$. Finally let us denote by $\lambda^X$ the normalized Lebesgue measure on a Borel set $X$ of $\R^d$ with $\lambda(X) > 0$ and by $\lambda^\beta_y$ the normalized Lebesgue measure on $\beta(y)$ for $y \in \hat U$ induced by Euclidean structure. By Fubini's theorem we have
\begin{align} \label{eq:CondMeasures1}
0 < \lambda^{\hat U}\left([I]\right) = \int_{[I]} \lambda^\beta_z([I] \cap \beta(z)) \dx \lambda^{\hat U}(z) = \int_{[I]} \lambda^\beta_z(\beta_I(z)) \dx \lambda^{\hat U}(z).
\end{align}
Because the submanifolds $\{\beta(z)\}_{z \in \hat U}$ are transversal the absolute continuity theorem (Theorem \ref{thm:ACT} {\it ii)}) implies that under the Poincar\'e map $P_{\beta(z),\beta(y)}$ the measures $\lambda^\beta_z$ and $\lambda^\beta_y$ are absolutely continuous for all $y,z \in [I]$. Thus $\lambda^\beta_y(\beta_I(y)) >0$ if and only if $\lambda^\beta_z(\beta_I(z)) >0$ for all $y,z \in [I]$ hence \rref{eq:CondMeasures1} finally implies $\lambda^\beta_y(\beta_I(y)) >0$ for all $y \in [I]$. Hence we can define the measure $\lambda^{\beta_I}_y := \lambda^\beta_z / \lambda^\beta_z(\beta_I(z))$ for $z \in [I]$. By $\lambda^\alpha_z$ we will denote the normalized Lebesgue measure on $\alpha(z)$, $z \in [I]$ induced by the Euclidean structure.

\begin{theorem} \label{thm:AbsContCondMeasures}
Let $(\omega,x) \in \Delta^l$. Denote by $\left\{\lambda^{[I]}_{\alpha(z)}\right\}_{z\in [I]}$ the canonical system of conditional measures of $\lambda^{[I]}$ associated with the measurable partition $\alpha$. Then for $\lambda$-almost every $z \in [I]$ the measure $\lambda^{[I]}_{\alpha(z)}$ is equivalent to $\lambda^\alpha_z$, moreover, we have
\begin{align*}
R^{-1}_{\Delta^l} \leq \frac{\dx \lambda^{[I]}_\alpha(z)}{\dx \lambda^\alpha_z} \leq R_{\Delta^l}
\end{align*}
$\lambda^\alpha_z$-almost everywhere on $\alpha(z)$, where $R_{\Delta^l} >0$ is a number depending only on the set $\Delta^l$ but not on the individual $(\omega,x) \in \Delta^l$.
\end{theorem}

\begin{proof}
The proof can be found in \citep[Theorem III.6.1]{Liu95}. Since for fixed $(\omega,x) \in \Delta^l$ this is a local property, the the proof remains true in our situation. Let us remark that an essential part of the proof is the absolute continuity theorem (Theorem \ref{thm:ACT}).
\end{proof}

%% file: Partition.tex
\section{Construction of the Partition} \label{sec:ConstructionPartition}

Recall that $\hat \Lambda_0 \subset \Omega^\N \times \R^d$ is the $F$-invariant set of full measure defined in \rref{eq:hatLambda0} then let us define
\begin{align*}
\hat\Lambda_1 := \{ (\omega,x) \in \hat\Lambda_0 : \lambda^{(1)}(x) < 0 \}
\end{align*}
and state two following definitions.

\begin{definition}
 A measurable partition $\eta$ of $\Omega^\N \times \R^d$ is said to be subordinate to $W^s$-submanifolds of $\mathcal{X}^+(\R^d, \nu,\mu)$, if for $\nu^\N \times \mu$-a.e. $(\omega,x)$, $\eta_\omega(x) := \{y : (\omega,y) \in \eta(\omega,x)\} \subset W^s(\omega,x)$ and it contains an open neighborhood of $x$ in $W^s(\omega,x)$, this neighborhood being taken in the submanifold topology of $W^s(\omega,x)$.
\end{definition}

\begin{definition}
We say that the Borel probability measure $\mu$ has absolutely continuous conditional measures on $W^s$-manifolds  of $\mathcal{X}^+(\R^d, \nu,\mu)$, if for any measurable partition $\eta$ subordinate to $W^s$-manifolds of $\mathcal{X}^+(\R^d, \nu,\mu)$ one has for $\nu^\N$-a.e. $\omega \in \Omega^\N$
\begin{align*}
\mu^{\eta_\omega}_x \ll \lambda^s_{(\omega,x)}, \qquad \mu-\text{a.e. } x \in \R^d
\end{align*}
where $\{\mu^{\eta_\omega}_x\}_{x \in \R^d}$ is a (essentially unique) canonical system of conditional measures of $\mu$ associated with the partition $\{\eta_\omega(x)\}_{x\in \R^d}$ of $\R^d$, and $\lambda^s_{(\omega,x)}$ is the Lebesgue measure on $W^s(\omega,x)$ induced by Euclidean structure as a submanifold of $\R^d$, where $\lambda^s_{(\omega,x)} = \delta_x$ if $(\omega,x) \notin \hat \Lambda_1$.
\end{definition}

Now we are able to state the main proposition, which yields a measurable partition $\eta$ with certain properties by which we are able to show the estimate of the entropy from below as presented in the next section.

\begin{proposition} \label{prop:mainprop}
Let $\mathcal{X}^+(\R^d, \nu, \mu)$ be given. Then there exists a measurable partition $\eta$ of $\Omega^\N \times \R^d$ which has the following properties:
\begin{enumerate}
\item $F^{-1}\eta \leq \eta$ and $\{\omega\} \times \R^d \leq \eta$;
\item $\eta$ is subordinate to $W^s$-manifolds of $\mathcal{X}^+(\R^d, \nu, \mu)$;
\item for every Borel set $B \in \mathcal{B}(\Omega^\N \times \R^d)$ the function
	\begin{align*}
	P_B(\omega, x) = \lambda^s_{(\omega,x)}(\eta_\omega(x) \cap B_\omega)
	\end{align*}
is measurable and $\nu^\N \times \mu$ almost everywhere finite, where $B_\omega := \{y: (\omega,y) \in B\}$ is the $\omega$-section of $B$;
\item if $\mu \ll \lambda$, then for $\nu^\N \times \mu$-a.e. $(\omega,x)$
	\begin{align*}
	\mu_x^{\eta_\omega} \ll \lambda^s_{(\omega,x)}.
	\end{align*}
\end{enumerate}
\end{proposition}

From Section \ref{eq:hatLambda0} we know that there exist countably many compact sets $\{\Lambda_i : \Lambda_i \subset \hat \Lambda_1\}_{i\in\N}$ such that $\nu^\N\times\mu(\hat\Lambda_1\!\setminus\!\bigcup_i\Lambda_i) = 0$ and each set $\Lambda_i$ is a set of type $\Delta^l$ as considered in Section \ref{sec:TheoremAbsoluteContinuity} and \ref{sec:AbsContCondMeasures} but with $\sTS{0} = \bigcup_{\lambda^{(j)}(x) < 0} V^{(j)}_{(\omega,x)}$ for each $(\omega,x) \in \Lambda_i$, i.e. $b = 0$. For $\Lambda_i \in \{\Lambda_i : i \in \N \}$ we will use the constants as in the previous sections, i.e. set $k_{\Lambda_i} := \dimension \sTS{0}$ for $(\omega,x) \in \Lambda_i$ and in the same way $A_{\Lambda_i}, \delta_{\Lambda_i}, q_{\Lambda_i}$ and so on. As in the previous sections we will denote the continuous family of $C^1$ embedded $k_{\Lambda_i}$-dimensional discs (the local stable manifolds) given by Theorem \ref{thm:localStableManifold} corresponding to $n=0$ by $\left\{W^s_{loc}(\omega,x)\right\}_{(\omega,x)\in\Lambda_i}$.

By Theorem \ref{thm:localStableManifold} there exist $\lambda_i > 0$ and $\gamma_i > 0$ such that for every $(\omega,x) \in \Lambda_i$, if $y,z \in W^s_{loc}(\omega,x)$ then for all $l \geq 0$ we have
\begin{align} \label{eq:ConvOnStblManf}
d^s(f^l_\omega y, f^l_\omega z) \leq \gamma_i e^{-\lambda_i l} d^s(y,z).
\end{align}

For $(\omega,x) \in \Lambda_i$ and $r > 0$ let us denote
\begin{align*}
B_{\Lambda_i}((\omega,x),r) := \left\{ (\omega',x') \in \Lambda_i : d(\omega,\omega') < r, \abs{x -x'} < r \right\},
\end{align*}
where as before $d$ denotes the metric on $\Omega^\N$ as introduced in Section \ref{sec:rds} and, to repeat, for $x \in \R^d$ and $(\omega,x) \in \Lambda_i$ respectively
\begin{align*}
B(x,r) &:= \{y \in \R^d : \abs{x -y} < r\}\\
\tilde U_{\Lambda_i,\omega}(x,r) &:= \exp_x \{\zeta \in T_x\R^d : \Lnorm[0]{\zeta} < r\}.
\end{align*}
Then we have the following corollary, which is an immediate consequence of Lemma \ref{lem:EstimatesOnLnorm} and Theorem \ref{thm:localStableManifold}.

\begin{corollary} \label{cor:ConclOfStblManf}
There exist numbers $r_i > 0$, $R_i > 0$ and $0 < \eps_i < 1$ such that the following hold true:
\begin{enumerate}
\item Let $(\omega,x) \in \Lambda_i$. If $(\omega',x') \in B_{\Lambda_i}((\omega,x),r_i)$ then
\begin{align*}
 B(x,r_i) \subset U_{\Lambda_i,\omega'}(x',q_{\Lambda_i}/2).
\end{align*}

\item For any $r \in [r_i/2, r_i]$ and each $(\omega,x) \in \Lambda_i$, if $(\omega', x') \in B_{\Lambda_i}((\omega,x),\eps_i r)$ then the local stable manifold $W^s_{loc}(\omega',x') \cap B(x,r)$ is connected and the map
\begin{align*}
 (\omega',x') \mapsto W^s_{loc}(\omega',x') \cap B(x,r)
\end{align*}
is continuous from $B_{\Lambda_i}((\omega,x),\eps_i r)$ to the space of subsets of $B(x,r)$ (endowed with the Hausdorff topology).

\item Let $r \in [r_i/2, r_i]$ and $(\omega, x) \in \Lambda_i$. If $(\omega',x'), (\omega', x'') \in B_{\Lambda_i}((\omega,x),\eps_i r)$ then either
\begin{align*}
 W^s_{loc}(\omega',x') \cap B(x,r) =  W^s_{loc}(\omega',x'') \cap B(x,r)
\end{align*}
or the two terms in the above equation are disjoint. In the latter case, if it is assumed moreover that $x'' \in W^s(\omega',x')$, then
\begin{align*}
 d^s(y,z) > 2r_i
\end{align*}
for any $y \in  W^s_{loc}(\omega',x') \cap B(x,r)$ and $z \in W^s_{loc}(\omega',x'') \cap B(x,r)$.

\item For each $(\omega,x) \in \Lambda_i$, if $(\omega',x') \in B_{\Lambda_i}((\omega,x),r_i)$ and $y \in W^s_{loc}(\omega',x') \cap B(x,r_i)$, then $W^s_{loc}(\omega',x')$ contains the close ball of center $y$ and $d^s$ radius $R_i$ in $W^s(\omega',x')$.
\end{enumerate}
\end{corollary}

\begin{proof}
Property {\it i)} is an immediate consequence of Lemma \ref{lem:EstimatesOnLnorm}. Where as properties {\it ii) - iv)} follows directly from Theorem \ref{thm:localStableManifold} and the choice of $q_{\Lambda_i}$ in Section \ref{sec:TheoremAbsoluteContinuity}.
\end{proof}

For the proof of Proposition \ref{prop:mainprop} we need some characterization of the $F$-invariant sets in terms of stable manifolds. Let us define
\begin{align*}
\mathcal{B}^s := \left\{B \in \mathcal{B}_{\nu^\N \times \mu}(\Omega^\N \times \R^d) : B = \bigcup_{(\omega,x) \in B} \{\omega\} \times W^s(\omega,x)\right\},
\end{align*}
where $\mathcal{B}_{\nu^\N \times \mu}(\Omega^\N \times \R^d)$ is the completion of $\mathcal{B}(\Omega^\N \times \R^d)$ with respect to $\nu^\N \times \mu$. Further denote the $\sigma$-algebra of $F$-invariant sets by
\begin{align*}
\mathcal{B}^I := \left\{ A \in \mathcal{B}_{\nu^\N \times \mu}(\Omega^\N \times \R^d) : F^{-1}A = A \right\}.
\end{align*}

Then we have the following lemma, which is \citep[Lemma IV.2.2]{Liu95} and states that every $F$-invariant set is basically a union of global stable manifolds.

\begin{lemma} \label{lem:BIinBs}
We have $\mathcal{B}^I \subset \mathcal{B}^s, \nu^\N \times \mu$-mod $0$.
\end{lemma}

\begin{proof}
The proof of \citep[Lemma III.2.2]{Liu95} is adapted to the case of $\R^d$, but follows along the same line. Put $\Omega^\N \times \mathcal{B}_\mu(\R^d) := \{\Omega^\N \times B: B \in \mathcal{B}_\mu(\R^d)\}$ where $\mathcal{B}_\mu(\R^d)$ is the completion of $\mathcal{B}(\R^d)$ with respect to $\mu$. Since the infinitely often differentiable functions with compact support on $\R^d$ are dense in $L^2(\R^d, \mathcal{B}(\R^d), \mu)$ and build a separable space there exists a countable set
\begin{align*}
\mathcal{F} := \{g_i: \text{the map } &g_i :\Omega^\N \times \R^d \to \R \text{ is continuous with compact support and}\\
&g_i(\omega,x) \text{ depends only on $x$ for each } (\omega,x) \in \Omega^\N \times \R^d, i \in \N\},
\end{align*}
which is dense in $L^2(\Omega^\N\times \R^d, \Omega^\N \times \mathcal{B}_\mu(\R^d),\nu^\N \times \mu)$. By Birkhoff's ergodic theorem for each $g_i \in \mathcal{F}$ there exists a set $\Lambda_{g_i} \in \mathcal{B}^I$ with $\nu^\N \times \mu(\Lambda_{g_i}) = 1$ such that for all $(\omega,x) \in \Lambda_{g_i}$ we have
\begin{align*}
\lim_{n\to\infty} \frac{1}{n} \sum_{k=0}^{n-1} g_i \circ F^k(\omega,x) = \E{g_i \big| \mathcal{B}^I}(\omega,x).
\end{align*}
Denote $\Lambda_\mathcal{F} := \bigcap_i \Lambda_{g_i}$. For two points $(\omega,y),(\omega,z) \in \Lambda_\mathcal{F}$ belonging to the same stable manifold, i.e. there exists $(\omega,x)$ such that $(\omega,y),(\omega,z) \in \{\omega\} \times W^s(\omega,x)$, we have $\lim_{n\to\infty} \abs{f^n_\omega y - f^n_\omega z} = 0$. Thus for any $\eps > 0$ and any $g_i \in \mathcal{F}$ there exists some compact set $C \subset \R^d$ and $\delta > 0$ such that $g_i\big|_{C^c} =0$ and $\abs{z-y} \leq \delta$ implies $\abs{g_i(z)-g_i(y)} \leq \eps$. Hence there exists $N \in \N$ such that we have
\begin{align*}
&\abs{\E{g_i \big| \mathcal{B}^I}(\omega,y) - \E{g_i \big| \mathcal{B}^I}(\omega,z)}
= \lim_{n\to\infty} \abs{\frac{1}{n} \sum_{k=0}^{n-1} \big(g_i(F^k(\omega,y)) - g_i(F^k(\omega,z))\big)}\\
&\hspace{10ex}\leq \lim_{n\to\infty} \frac{1}{n} \sum_{k=0}^{N-1} \abs{g_i(F^k(\omega,y)) - g_i(F^k(\omega,z))} + \lim_{n\to\infty} \frac{n-N}{n}\eps\\
&\hspace{10ex}= \eps.
\end{align*}
Since $\eps > 0$ can be chosen arbitrarily small we have $\E{g_i \big| \mathcal{B}^I}(\omega,y) = \E{g_i \big| \mathcal{B}^I}(\omega,z)$ for $(\omega,y)$ and $(\omega,z)$ on the same stable manifold.
Hence for all $i \in \N$ the conditional expectation $\E{g_i \big| \mathcal{B}^I}\!\!\big|_{\Lambda_\mathcal{F}}$ restricted to $\Lambda_\mathcal{F}$ is measurable with respect to $\mathcal{B}^s|_{\Lambda_\mathcal{F}}$, which implies
\begin{align} \label{eq:BIinBs}
\left\{\E{g_i \big| \mathcal{B}^I}\!\!\big|_{\Lambda_\mathcal{F}} : g_i \in \mathcal{F}\right\} \subset L^2(\Lambda_\mathcal{F}, \mathcal{B}^s|_{\Lambda_\mathcal{F}}, \nu^\N \times \mu).
\end{align}
Since the functions that are invariant with respect to $F$ do not depend on $\omega$ (see \citep[Corollary I.1.1]{Liu95}) we have
\begin{align*}
L^2(\Omega^\N\times \R^d, \mathcal{B}^I,\nu^\N \times \mu) \subset L^2(\Omega^\N\times \R^d, \Omega^\N \times \mathcal{B}_\mu(\R^d),\nu^\N \times \mu).
\end{align*}
Since $\mathcal{F}$ is a dense subset of the right-hand space and the conditional expectation can be seen as an orthogonal projection we have that $\left\{\E{g_i\big|\mathcal{B}^I} : g_i\in \mathcal{F}\right\}$ is dense in $L^2(\Omega^\N\times \R^d, \mathcal{B}^I,\nu^\N \times \mu)$. Then from \rref{eq:BIinBs} it follows that
\begin{align*}
L^2(\Lambda_\mathcal{F}, \mathcal{B}^I|_{\Lambda_\mathcal{F}}, \nu^\N \times \mu) \subset L^2(\Lambda_\mathcal{F}, \mathcal{B}^s|_{\Lambda_\mathcal{F}}, \nu^\N \times \mu),
\end{align*}
which implies since $\nu^\N \times \mu(\Lambda_\mathcal{F}) =1$ the desired, i.e.
\begin{align*}
 \mathcal{B}^I \subset \mathcal{B}^s, \nu^\N \times \mu\text{-mod } 0.
\end{align*}
\end{proof}

Let us now state the a sketch of the proof of Proposition \ref{prop:mainprop}, which is \citep[Proposition IV.2.1]{Liu95}, in particular the construction of the partition $\eta$.

\begin{proof}[Proof of Proposition \ref{prop:mainprop}]
{\it Step 1.} Let $\Lambda_i \in \left\{ \Lambda_i, i \in \N \right\}$ be arbitrarily fixed and choose the constants $\eps_{i}, r_{i}$ and $R_{i}$ according to Corollary \ref{cor:ConclOfStblManf}. Since $\Lambda_i$ is compact, the open cover $\left\{ B_{\Lambda_i}((\omega,x),\eps_i r_i/2) \right\}_{(\omega,x)\in \Lambda_i}$ has a finite subcover $\mathcal{U}_{\Lambda_i}$ of $\Lambda_i$. Let us fix arbitrarily $B_{\Lambda_i}((\omega_0,x_0),\eps_i r_i/2) \in \mathcal{U}_{\Lambda_i}$. For each $r \in [r_i/2,r_i]$ we define
\begin{align*}
S_r := \bigcup_{(\omega,x) \in B_{\Lambda_i}((\omega_0,x_0),\eps_i r)} \left\{\{\omega\} \times [W^s_{loc}(\omega,x) \cap B(x_0,r)]\right\}. 
\end{align*}
Denote by $\xi_r$ the partition of $\Omega^\N \times \R^d$ into all sets $\{\omega\} \times [W^s_{loc}(\omega,x) \cap B(x_0,r)]$, $(\omega,x) \in B_{\Lambda_i}((\omega_0,x_0),\eps_i r)$ and the set $\Omega^\N \times \R^d\!\setminus\! S_r$. By $ii)$ and $iii)$ of Corollary \ref{cor:ConclOfStblManf} one sees that $\xi_r$ is a partition and by the continuity property of the local stable manifolds that it is even a measurable partition. Now put
\begin{align*}
\eta_r := \left( \bigvee_{n=0}^{+\infty} F^{-n}\xi_r \right) \vee \left\{\{\omega\} \times \R^d : \omega \in \Omega^\N\right\}.
\end{align*}
One can see (\citep[Proof of IV.2.1]{Liu95}) that for almost every $r\in[r_i/2,r_i]$ the partition $\eta_r$ has the following properties:
\begin{enumerate}
\item[(1)] $F^{-1}\eta_r \leq \eta_r$ and $\left\{\{\omega\} \times \R^d : \omega \in \Omega^\N\right\} \leq \eta_r$;
\item[(2)] Put $\hat S_r = \bigcup_{n=0}^{+\infty} F^{-n}S_r$. Then for $\nu^\N \times \mu$-a.e. $(\omega,y) \in \hat S_r$ we have $(\eta_r)_\omega(y) := \{z : (\omega,z) \in \eta_r(\omega,y)\} \subset W^s(\omega,y)$ and it contains an open neighborhood of $y$ in $W^s(\omega,y)$;
\item[(3)] For any $B \in \mathcal{B}(\Omega^\N \times \R^d)$ the function
\begin{align*}
P_B(\omega,y) = \lambda^s_{(\omega,y)}((\eta_r)_\omega(y) \cap B_\omega)
\end{align*}
is measurable and finite $\nu^\N \times \mu$-a.e. on $\hat S_r$;
\item[(4)] Define $\hat\eta_r = \eta_r\big|_{\hat S_r}$ and for $\omega \in \Omega^\N$ let $\{\mu_{(\hat\eta_r)_\omega(y)}\}_{y\in (\hat S_r)_\omega}$ be a canonical system of conditional measures of $\mu\big|_{(\hat S_r)_\omega}$ associated with the partition $(\hat\eta_r)_\omega$. If $\mu \ll \lambda$ then for $\nu^\N$-a.e. $\omega \in \Omega^\N$ it holds that
\begin{align*}
\mu_{(\hat\eta_r)_\omega(y)} \ll \lambda^s_{(\omega,y)} \quad \mu\text{-a.e. } y \in (\hat S_r)_\omega.
\end{align*}
\end{enumerate}
Let us remark that for the proof of property (4) Theorem \ref{thm:AbsContCondMeasures} is the essential part.

{\it Step 2.} 
Let us notice that Step 1 works for any $\Lambda^i$ and any set in $\mathcal{U}_{\Lambda^i}$. So let us denote $\bigcup_{i=1}^{+\infty} \mathcal{U}_{\Lambda^i} = \{U_1,U_2,U_3,\dots\}$ and for each $U_n$ we will denote the partition $\eta_r$ satisfying (1)-(4) from Step 1 by $\eta_n$ and the associated set $\hat S_r$ by $\hat S_n$. Define for each $n \geq 0$ the set $I_n := \bigcap_{l=1}^{+\infty} F^{-l}\hat S_n$. Then we have
\begin{align*}
I_n = \bigcap_{l =1}^{+\infty} \bigcup_{k\geq l} F^{-k} S_n
\end{align*}
and thus clearly $F^{-1} I_n = I_n$. The Poincar{\'e} recurrence theorem implies $\nu^\N\times\mu(\hat\Lambda_1\setminus \bigcup_{n=1}^{+\infty} I_n) = 0$. Because of Lemma \ref{lem:BIinBs} we can and will assume that $I_n \in \mathcal{B}^s$. If this is not the case we would proceed with $I_n' \in \mathcal{B}^s$ such that $F^{-1} I_n' = I_n'$ and $\nu^\N\times\mu(I_n \triangle I_n') = 0$. So let us now define $\hat \eta_n := \eta_n|_{I_n}$. Since $I_n \in \mathcal{B}^s$ we have
\begin{align*}
I_n = \bigcup_{(\omega',x')\in I_n} \{\omega'\}\times W^s(\omega',x').
\end{align*}
and thus
\begin{align} \label{eq:PreservingPropOfEta}
\hat \eta_n = \left\{\eta_n(\omega,x) \cap I_n\right\}_{(\omega,x)\in I_n}
= \left\{\eta_n(\omega,x) \cap \{\omega\}\times W^s(\omega,x)\right\}_{(\omega,x)\in I_n},
\end{align}
which implies that $\hat \eta_n$ preserves the structure of $\eta_n$ as constructed in Step 1. So let us define finally the partition $\eta$ of $\Omega^\N \times \R^d$ by
\begin{align*}
\eta(\omega,x) =
\begin{cases}
\hat \eta_1(\omega,x), &\quad \text{if } (\omega,x) \in I_1 \\
\hat \eta_n(\omega,x), &\quad \text{if } (\omega,x) \in I_n\! \setminus \!\bigcup_{k=1}^{n-1} I_k \\
\{(\omega,x)\}, &\quad \text{if } (\omega,x) \in \Omega^\N \times \R^d \!\setminus\! \bigcup_{n=1}^{+\infty} I_n \\
\end{cases}
\end{align*}
Because by \rref{eq:PreservingPropOfEta} we have for $(\omega,x) \in I_n\!\setminus\! \bigcup_{k=1}^{n-1} I_k$ for some $n\geq 1$ that $\eta(\omega,x) = \eta_n(\omega,x)$ and thus clearly satisfies property (1) and properties (2)-(4) on $I_n$ instead of $\hat S_r$. Since $\nu^\N\times\mu(\hat \Lambda_1\!\setminus\! \bigcup_{n=1}^{+\infty} I_n) = 0$ and for $(\omega,x)\notin\hat\Lambda_1$ we defined $W^s(\omega,x) = \{x\}$ and $\lambda^s_{(\omega,x)} = \delta_x$ the properties of Proposition \ref{prop:mainprop} are satisfied $\nu^\N\times\mu$-almost everywhere, which completes the proof.
\end{proof}

By Property iii) of Proposition \ref{prop:mainprop} we can define as in \citep[Section IV.2]{Liu95} a Borel measure $\lambda^*$ on $\Omega^\N\times\R^d$ by
\begin{align*}
\lambda^*(K) := \int \lambda^s_{(\omega,x)}(\eta_\omega(x) \cap K_\omega)\, \dx \nu^\N\times\mu(\omega,x)
\end{align*}
for any $K \in \mathcal{B}(\Omega^\N\times\R^d)$. One can easily see that $\lambda^*$ is a $\sigma$-finite measure. By definition of the canonical system of conditional measures we have
\begin{align*}
\nu^\N \times \mu(K) = \int \mu^{\eta_\omega}_x(\eta_\omega(x) \cap K_\omega)\, \dx\nu^\N\times\mu(\omega,x)
\end{align*}
for each $K \in \mathcal{B}(\Omega^\N\times\R^d)$. Since by Property iv) of Propostion \ref{prop:mainprop} for $\nu^\N\times\mu$-almost every $(\omega,x) \in \Omega^\N\times\R^d$ we have $\mu^{\eta_\omega}_x\ll\lambda^s_{(\omega,x)}$ we get
\begin{align*}
\nu^\N\times\mu \ll \lambda^*.
\end{align*}
So let us define
\begin{align*}
g := \frac{\dx \nu^\N \times \mu}{\dx \lambda^*}.
\end{align*}
Then we have the following proposition, which is \citep[Proposition IV.2.2]{Liu95}.

\begin{proposition} \label{prop:PropoertieOfg}
For $\nu^\N\times\mu$-almost every  $(\omega,x)$, we have
\begin{align}
g = \frac{\dx \mu^{\eta_\omega}_x}{\dx \lambda^s_{(\omega,x)}}
\end{align}
$\lambda^s_{(\omega,x)}$-a.e. on $\eta_\omega(x)$.
\end{proposition}

\begin{proof}
This is \citep[Proposition III.2.2]{Liu95}.
\end{proof}

%% file: below.tex
\section{Proof of Theorem \ref{thm:Pesin}} \label{sec:proofPesin}

In this section we will state the proof of Pesin's formula for random dynamical systems on $\R^{d}$ which have an invariant probability measure and satisfies the assumptions from Section \ref{sec:rds}.

\subsection{Estimation of the Entropy from Below} \label{sec:EstimateFromBelow}

First we will state the proof of the estimation of the entropy from below, i.e. the following the result, which is basically taken from \citep[Section IV.3]{Liu95} and bases on the partition constructed in the previous section.

\begin{theorem}
Let $\mathcal{X}(\R^d, \nu, \mu)$ be a random dynamical system that satisfies Assumptions \ref{ass1} - \ref{ass2b}. If the invariant measure $\mu$ is absolutely continuous with respect to Lebesgue measure on $\R^d$ then we have
\begin{align*}
h_\mu(\mathcal{X}(\R^d, \nu,\mu)) \geq \int \sum_i \lambda^{(i)}(x)^+ m_i(x) \dx \mu.
\end{align*}
\end{theorem}

\begin{proof}
Assuming for the moment that
\begin{align} \label{eq:assume1}
H_{\nu^\N \times \mu} (\eta|F^{-n}\eta \vee \sigma_0) < +\infty
\end{align}
then one can show (see \citep[Proof of Theorem IV.1.1]{Liu95}) that by Theorems \ref{thm:EqualityOfEntropies1} and \ref{thm:EqualityOfEntropies2}
\begin{align*}
\lim_{n \to \infty} \frac{1}{n} H_{\nu^\N \times \mu} (\eta|F^{-n}\eta \vee \sigma_0) &\leq H_{\mu^*}(\eta^+ | G^{-1} \eta^+ \vee \sigma) = h_{\mu^*}^\sigma (G, \eta^+) \\
&\leq \sup_{\xi} h_{\mu^*}^\sigma (G, \xi) = h_{\mu^*}^\sigma (G) = h_\mu(\mathcal{X}(\R^d, \nu,\mu)),
\end{align*}
where $G$ was defined in Section \ref{sec:rds}, $\sigma_{0}$ and $\sigma$ were defined in Section \ref{subsec:EntropyOfRandDiffeos}, $\mu^{*}$ is the measure defined by Proposition \ref{prop:ExistenceOfMuStar} and $\eta^{+} := P^{-1}\eta$ with the projection $P$ as defined in Section \ref{subsec:EntropyOfRandDiffeos}. Thus it suffices to show that \rref{eq:assume1} is true and that for all $n \geq 1$
\begin{align} \label{eq:tobeshown1}
\frac{1}{n} H_{\nu^\N \times \mu} (\eta|F^{-n}\eta \vee \sigma_0) \geq \int \sum_i \lambda^{(i)}(x)^+ m_i(x) \dx \mu.
\end{align}

By \rref{eq:alternativeCondEntropy} and the properties of the partition $\eta$ we get
\begin{align} \label{eq:ReoresentationOfEntropy}
H_{\nu^\N \times \mu} (\eta|F^{-n}\eta \vee \sigma_0) &= - \int_{\Omega^\N \times \R^d} \log\left(\nu^\N \times \mu_{(\omega,x)}^{F^{-n}\eta \vee \sigma_0}(\eta(\omega,x))\right) \dx \nu^\N \times \mu(\omega,x) \notag \\
&= - \int_{\Omega^\N} \int_{\R^d} \log\left(\mu_{x}^{(f^n_\omega)^{-1}\eta_{\tau^n\omega}}(\eta_\omega(x))\right) \dx \mu(x) \dx\nu(\omega).
\end{align}
Let $\{I_j\}_{j \in \N}$ be the sets from the proof of Proposition \ref{prop:mainprop} of the construction of the partition $\eta$ and define $I := \bigcup_{j \in \N} I_j$ and $I_0 := \Omega^\N \times\R^d \!\setminus\! I$. Since each $I_j$ is $F$-invariant we have $F^{-1}I = I$ and $F^{-1}I_0 = I_0$. Thus $\eta$ and $F^{-n}\eta \vee \sigma_0$ are refinements of the partition $\{I,I_0\}$ and their restriction to $I_0$ is the partition into single points which implies for each $(\omega,x) \in I_{0}$
\begin{align*}
\log\left(\mu_{x}^{(f^n_\omega)^{-1}\eta_{\tau^n\omega}}(\eta_\omega(x))\right) = 0.
\end{align*}
By definition of $\hat\Lambda_1$ the Lyapunov exponents are all non-negative on $(\Omega^\N\times\R^d) \!\setminus\!\hat\Lambda_1$, i.e. $\lambda^{(1)}(x) \geq 0$ on $(\Omega^\N\times\R^d) \!\setminus\!\hat\Lambda_1$. Thus we get from \citep[Proposition I.3.3]{Liu95} 
\begin{align*}
0 \leq \int_{I_0} \sum_i \lambda^{(i)}(x)^+ m_i(x) \dx \nu^\N\times\mu = \int_{I_0} \sum_i \lambda^{(i)}(x) m_i(x) \dx \nu^\N\times\mu \leq 0,
\end{align*}
which implies
\begin{align*}
\int_{I_0} \sum_i \lambda^{(i)}(x)^+ m_i(x) \dx \nu^\N\times\mu = 0.
\end{align*}
So in the following let us assume without  loss of generality that $\nu^\N\times\mu(I) = 1$.

Because of $\mu \ll \lambda$ and the invariance of $\mu$ we get that there exists a Borel subset $\Gamma' \subset \Omega^\N$ with $\nu^\N(\Gamma') = 1$ such that for any $\omega \in \Gamma'$
\begin{align*}
\mu\ll \mu \circ f^n_\omega,
\end{align*}
where $\mu \circ f^n_\omega(E) := \mu(f^n_\omega(E))$ for any Borel set $E \subset \R^d$. Denoting by $\phi := \dx \mu / \dx \lambda$ the Radon-Nikodym derivative it is easy to check that for any $\omega \in \Gamma'$
\begin{align*}
\frac{\dx\mu}{\dx (\mu \circ f^n_\omega)}(z) = \frac{\phi(z)}{\phi(f^n_\omega z)} \abs{\determinante D_z f^n_\omega}^{-1} =: \Phi_n(\omega,z).
\end{align*}
Then Proposition \ref{prop:AbsCondMeasures} implies that
\begin{align*}
\frac{\dx \mu^{(f^n_\omega)^{-1}\eta_{\tau^n \omega}}_x}{\dx (\mu \circ f^n_\omega)^{(f^n_\omega)^{-1}\eta_{\tau^n \omega}}_x} = \frac{\Phi_n(\omega, \cdot)|_{(f^n_\omega)^{-1}\eta_{\tau^n \omega}(x)}}{\int_{(f^n_\omega)^{-1}\eta_{\tau^n \omega}(x)} \Phi_n(\omega,z) \dx (\mu \circ f^n_\omega)^{(f^n_\omega)^{-1}\eta_{\tau^n \omega}}_x}
\end{align*}
for $\mu$-a.e. $x \in \R^d$. For $\nu^\N \times \mu$-a.e. $(\omega,y) \in \Omega^\N \times \R^d$ let us define
\begin{align*}
W_n(\omega,x) &:= \mu_{y}^{(f^n_\omega)^{-1}\eta_{\tau^n\omega}}(\eta_\omega(y))\\
X_n(\omega,x) &:= \frac{\phi(y)}{\phi(f^n_\omega y)} \frac{g(F^n(\omega,y)}{g(\omega,y)} \\
Y_n(\omega,x) &:= \frac{\abs{\determinante(D_yf^n_\omega|_{E_0(\omega,z)})}}{\abs{\determinante(D_yf^n_\omega)}}\\
Z_n(\omega,x) &:= \int_{(f^n_\omega)^{-1}\eta_{\tau^n \omega}(y)} \Phi_n(\omega,z) \dx (\mu \circ f^n_\omega)^{(f^n_\omega)^{-1}\eta_{\tau^n \omega}}_y,
\end{align*}
where $g$ is the function defined before Proposition \ref{prop:PropoertieOfg}. Then one can show (see \citep[Claim IV.3.1]{Liu95}) using change of variables formula twice and the absolute continuity of $\mu \ll \lambda$ and $\mu^{\eta_\omega}_x \ll \lambda^s_{(\omega,x)}$ for $\nu^\N\times\mu$-a.e. $(\omega,x)$ that almost everywhere on $\Omega^\N\times\R^d$ we have
\begin{align} \label{eq:RepresentationOfW}
W_n(\omega,x) = \frac{X_n(\omega,x) Y_n(\omega,x)}{Z_n(\omega,x)}.
\end{align}

Because of
\begin{align*}
\abs{\determinante(D_xf^n_\omega))} \leq \abs{D_xf^n_\omega}^{d}
\end{align*}
Assumption \ref{ass1} implies for each $n \geq 1$ that $\log^{+} \abs{\determinante(D_xf^n_\omega)} \in \mathcal{L}^1(\nu^\N\times\mu)$ and analogously that $\log^{+} \abs{\determinante(D_xf^n_\omega|_{E_0(\omega,x)})} \in \mathcal{L}^1(\nu^\N\times\mu)$. Thus by the multiplicative ergodic theorem we have for $n\geq 1$
\begin{align} \label{eq:detequation1}
\frac{1}{n} \int \log \abs{\determinante(D_xf^n_\omega)} \dx\nu^\N\times\mu = \int \sum_i \lambda^{(i)}(x) m_i(x) \dx \mu(x)
\end{align}
and
\begin{align} \label{eq:detequation2}
\frac{1}{n} \int \log \abs{\determinante(D_xf^n_\omega|_{E_0(\omega,x)})} \dx\nu^\N\times\mu = \int \sum_i \lambda^{(i)}(x)^- m_i(x) \dx \mu(x),
\end{align}
where both sides of the two equations might be $-\infty$. By the multiplicity of the determinante  Assumption \ref{ass2b} implies that $\log \abs{\determinante(D_xf^n_\omega)} \in \mathcal{L}^1(\nu^\N\times\mu)$ for $n\geq 1$ and thus by \rref{eq:detequation1} that
\begin{align*}
\sum_i \lambda^{(i)}(x) m_i(x) \in \mathcal{L}^{1}(\R^{d},\mu).
\end{align*}
This yields by \rref{eq:detequation2} that $\log \abs{\determinante(D_xf^n_\omega|_{E_0(\omega,x)})} \in \mathcal{L}^1(\nu^\N\times\mu)$, which finally implies the integrability of $\log Y_{n}$, i.e. $\log Y_n \in \mathcal{L}^1(\nu^\N \times \mu)$ and
\begin{align} \label{eq:EstimateOnY}
- \frac{1}{n} \int \log Y_n \dx \nu^\N \times \mu = \int \sum_i \lambda^{(i)}(x)^+ m_i(x) \dx \mu.
\end{align}

Further from \citep[Claim IV.3.3 and IV.3.4]{Liu95} we get that $\log X_n \in \mathcal{L}^1(\nu^\N \times \mu)$ and $\log Z_n \in \mathcal{L}^1(\nu^\N \times \mu)$ with
\begin{align} \label{eq:EstimateOnX}
- \frac{1}{n} \int \log X_n \dx \nu^\N \times \mu = 0
\end{align}
and
\begin{align} \label{eq:EstimateOnZ}
- \frac{1}{n} \int \log Z_n \dx \nu^\N \times \mu \geq 0.
\end{align}
Combining now \rref{eq:EstimateOnY}, \rref{eq:EstimateOnX} and \rref{eq:EstimateOnZ} via \rref{eq:RepresentationOfW} and \rref{eq:ReoresentationOfEntropy} finishes the proof.
\end{proof}

\subsection{Estimate of the Entropy from Above} \label{sec:EstimateFromAbove}

A nice an short proof  of the reverse inequality was given in \citep{bahnmueller95} for random dynamical systems on a compact Riemannian manifold. This proof was extended in \citep{bargen10} to isotropic Ornstein-Uhlenbeck flows, which can be seen as some special random dynamical system on $\R^d$. This proof can be extended to our more general situation assuming Assumption \ref{ass3}. Precisely we have the following theorem.

\begin{theorem} \label{thm:EstimateFromAbove}
Let $\mathcal{X}(\R^d, \nu, \mu)$ be a random dynamical system that satisfies Assumption \ref{ass1} and \ref{ass3}, then we have
\begin{align*}
h_\mu(\mathcal{X}(\R^d, \nu)) \leq \int \sum_i \lambda^{(i)}(x)^+ m_i(x) \dx \mu.
\end{align*}
\end{theorem}

\begin{proof}
Let us remark that for isotropic Ornstein-Uhlenbeck flows the distribution of the derivative is translation invariant. Thus for $k \in \N$, $\omega \in \Omega^\N$ and $y \in \R^d$ the random variable
\begin{align*}
L_k(n,\omega,y) := \sup_{z \in B(y,\frac{1}{k})} \abs{D_z f^n_\omega},
\end{align*}
is independent of $y$ and hence
\begin{align} \label{eq:holger1}
\int_{\Omega^\N} \log^+(L_1(n,\omega,y)) \dx \nu^\N(\omega)
\end{align}
is uniformly bounded in $y \in \R^d$. Since we clearly do not have the translation invariance for any random dynamical system we need to have a closer look at the two estimates in \citep{bargen10} where \rref{eq:holger1} is used. In particular we need to bound
\begin{align*}
\lim_{k \to \infty} \sum^{+\infty}_{i=m+1} \mu(\xi_{x_{i}}) \int_{\Omega^{\N}} \log^{+}(L_{k}(n,\omega,x_{i}))\,\dx\nu^{\N}(\omega)
\end{align*}
for the estimate on term $II$ and show that
\begin{align} \label{eq:holger2}
\lim_{k \to \infty} \sum^{m}_{i=1} \mu(\xi_{x_{i}}) \int_{\Omega^{\N}\setminus \Omega_{k,l}} \log^{+}(L_{k}(n,\omega,x_{i}))\,\dx\nu^{\N}(\omega) = 0
\end{align}
for the estimate on term $III$, where for each $k,l \in \N$ the family of sets $\{\xi_{x_{i}}\}_{i=1,\dots,m}$ is a partition of $B(0,l)$ and $\{\xi_{x_{i}}\}_{i\geq m+1}$ a partition of $\R^{d}\!\setminus\!B(0,l)$ with $\xi_{x_{i}} \subset B(x_{i},1/k)$ for every $i \in \N$. The sets $\Omega_{k,l}$ are certain subsets of $\Omega^{\N}$ such that for each fixed $l\in \N$ we have $\Omega_{k,l} \nearrow \Omega$ for $k \to \infty$. For details concerning the definition of $\{\xi_{x_{i}}\}_{i\in \N}$ and $\Omega_{k,l}$ see \citep{bargen10}. Then for any $i \in \N$ and $x \in \xi_{x_{i}}$ we have
\begin{align*}
B\left(x_{i},\frac{1}{k}\right) \subset B\left(x,\frac{2}{k}\right).
\end{align*}
Thus we get by monotonicity of $\log^{+}$
\begin{align*}
\lim_{k \to \infty} \sum^{+\infty}_{i=m+1} \mu(\xi_{x_{i}}) &\int_{\Omega^{\N}} \log^{+}(L_{k}(n,\omega,x_{i}))\,\dx\nu^{\N}(\omega)\\
&\leq \lim_{k \to \infty} \sum^{+\infty}_{i=m+1} \int_{\xi_{x_{i}}}\int_{\Omega^{\N}} \log^{+}(L_{k/2}(n,\omega,x))\,\dx\nu^{\N}(\omega)\dx\mu(x)\\
&\leq\int_{\R^{d}\setminus B(0,l)} \int_{\Omega^{\N}} \log^{+}(L_{1}(n,\omega,x))\,\dx\nu^{\N}(\omega)\dx\mu(x)\\
&=\int_{\R^{d}\setminus B(0,l)} \int_{\Omega^{\N}} \sup_{z\in B(x,1)}\log^{+}(\abs{D_{z}f^{n}_{\omega}})\,\dx\nu^{\N}(\omega)\dx\mu(x),
\end{align*}
which is finite because of Assumption \ref{ass3}. On the other hand we have analogously
\begin{align*}
\sum^{m}_{i=1} \mu(\xi_{x_{i}}) &\int_{\Omega^{\N}\setminus \Omega_{k,l}} \log^{+}(L_{k}(n,\omega,x_{i}))\,\dx\nu^{\N}(\omega)\\
&\leq \int_{B(0,l)} \int_{\Omega^{\N}\setminus\Omega_{k,l}} \sup_{z\in B(x,1)}\log^{+}(\abs{D_{z}f^{n}_{\omega}})\,\dx\nu^{\N}(\omega)\dx\mu(x).
\end{align*}
Because of Assumption \ref{ass3} and $\Omega_{k,l} \nearrow \Omega$ this last expression converges to $0$ for $k \to \infty$ by dominated convergence. By this the proof of Theorem \ref{thm:EstimateFromAbove} follows strictly along the proof in \citep{bargen10}.
\end{proof}

%% file: ApplicationFlows.tex
\section{Application to Stochastic Flows} \label{sec:ApplicationToFlows}

In this section we will show that a broad class of stochastic flows which are generated by stochastic differential equations driven by continuous semimartingale noise can be seen as a random dynamical system in the sense of Section \ref{sec:rds}. Assuming that the generated dynamical system has an invariant probability measure that satisfies a mild integrability assumptions, we will show that the assumptions of Section \ref{sec:rds} are satisfied and hence Pesin's formula holds.

\subsection{Definition of Stochastic Flows}
For a short introduction to stochastic flows we will follow \citep[Section 2]{Imkeller99}. Let $\{F(x,t)\}_{t \geq 0}$ be a family of $\R^{d}$-valued continuous semimartingales indexed by $x \in \R^{d}$ on a filtered probability space $(\bar \Omega, \bar{\mathcal{F}}, (\bar{\mathcal{F}_{t}})_{t\geq 0}, \mathbf{P})$. Let $F(x,t) = M(x,t) + V(x,t)$ be the canonical decomposition of the semimartingale into a local martingale $M$ and a process $V$ of locally bounded variation. We will assume in the following that both $M$ and $V$ are jointly continuous in $(x,t)$ and furthermore that there exists $a: \R^{d} \times \R^{d} \times [0,+\infty) \times \bar \Omega \to \R^{d\times d}$ and $b: \R^{d} \times [0,+\infty) \times \bar \Omega \to \R^{d}$ such that
\begin{align*}
\langle M_{i}(x, \cdot), M_{j}(y, \cdot) \rangle_{t} = \int_{0}^{t} a_{i,j}(x,y,u) \dx u, \qquad V_{i}(x,t) = \int_{0}^{t} b_{i}(x,u) \dx u,
\end{align*}
where $\langle \cdot, \cdot\rangle_{t}$ denotes the quadratic variation process at time $t$. The functions $a$ and $b$ are called the local characteristics of $F$.

For a multi index $\alpha = (\alpha_{1}, \dots, \alpha_{d})$ with $\alpha_{i} \in \N_{0}, i = 1, \dots, d$ we write $\abs{\alpha} := \sum_{i=1}^{d} \abs{\alpha_{i}}$. We will say that $F$ has local characteristics of class $B^{m,\delta}_{ub}$ for $m \in \N_{0}$, $0 < \delta \leq 1$ (or just $F\in B^{m,\delta}_{ub}$) if $b \in C^{m}$ and all derivatives of $a$ up to order $m$ with respect to $x$ and $y$ (simultaneously) are continuous and for all $T > 0$
\begin{align*}
\esssup_{\bar\omega \in \bar \Omega} \sup_{0\leq t \leq T} \left(\norm{a(t)}^{\sim}_{m+\delta} + \norm{b(t)}_{m+\delta}\right) < +\infty,
\end{align*}
where
\begin{align*}
\norm{a(t)}^{\sim}_{m+\delta} &:= \sup_{x, y \in \R^{d}} \frac{\abs{a(x,y,t)}}{(1+\abs{x})(1 + \abs{y})} + \sum_{1 \leq \abs{\alpha} \leq m} \sup_{x, y \in \R^{d}} \abs{D^{\alpha}_{1}D^{\alpha}_{2}a(x,y,t)} \\
&\hspace{5ex}+ \sum_{\abs{\alpha} = m} \norm{D^{\alpha}_{1}D^{\alpha}_{2}a(\cdot,\cdot,t)}^{\sim}_{\delta},\\
\intertext{with}
\norm{f}^{\sim}_{\delta} &:= \sup_{x\neq x',y\neq y'}\frac{\abs{f(x,y) - f(x',y) - f(x,y') + f(x',y')}}{\abs{x-x'}^{\delta}\abs{y-y'}^{\delta}},\\
\intertext{and}
\norm{b(t)}_{m+\delta} &:= \sup_{x\in \R^{d}} \frac{\abs{b(x,t)}}{(1+\abs{x})} + \sum_{1 \leq \abs{\alpha} \leq m} \sup_{x\in \R^{d}} \abs{D^{\alpha}b(x,t)} \\
&\hspace{5ex}+ \sum_{\abs{\alpha} = m} \sup_{x\neq y}\frac{\abs{D^{\alpha}b(x,t) - D^{\alpha}b(y,t)}}{\abs{x-y}^{\delta}},
\end{align*}
where $D_{j}$ denotes the derivative operator with respect to the $j\textsuperscript{th}$ spatial variable and $D := D_{1}$ if there is only one spatial variable. If we consider the stochastic differential equation
\begin{align} \label{eq:sde}
\dx X(t) = F(X(t),\dx t)
\end{align}
on $\R^{d}$ where $F \in B^{k,\delta}_{ub}$ with $k \geq 1$ and $0 <\delta\leq 1$ is a spatial semimartingale as above then by \citep[Theorem 4.6.5]{Kunita90} there exists a stochastic flow of diffeomorphisms associated with \rref{eq:sde}. That is a map $\flow: [0,+\infty) \times [0,+\infty) \times \R^{d} \times \bar\Omega \to \R^{d}$ such that
\begin{enumerate}
\item $\flow_{s,t}(x,\cdot)$, $t\geq s$ solves \rref{eq:sde} with initial condition $X(s) = x$ for $s \geq 0$, $x \in \R^{d}$;
\item $\flow_{s,t}(\cdot,\bar\omega)$ is a $C^k$-diffeomorphism for each $s,t \geq 0$, $\bar\omega\in\bar\Omega$;
\item $\flow_{s,t}(\cdot,\bar\omega) = \flow_{t,s}^{-1}(\cdot,\bar\omega)$ for each $s,t \geq 0$, $\bar\omega \in \bar \Omega$;
\item $\flow_{s,u}(\cdot,\bar\omega) = \flow_{t,u}(\cdot,\bar\omega)\circ\flow_{s,t}(\cdot,\bar\omega)$ for each $s,t,u \geq 0$, $\bar\omega \in \bar \Omega$;
\item $(s,t) \mapsto \flow_{s,t}(\cdot,\bar\omega)$ is continuous from $[0,+\infty)^{2}$ to the (group of) diffeomorphisms on $\R^{d}$.
\end{enumerate}

If $F \in B^{k,\delta}_{ub}$ for some $k \geq 1$ and $0 < \delta \leq 1$ and the correction term
\begin{align*}
c(x,t) := \sum_{j=1}^d \frac{\partial a^{\cdot j}}{\partial x_j}(x,y,t)\bigg|_{y=x}
\end{align*}
also belongs to $B^{k,\delta}_{ub}$ then the generating semimartingale field of the backward flow $\{\flow_{t,s} : 0\leq s \leq t < \infty\}$ is also an element of $B^{k,\delta}_{ub}$ (see \citep[Section 4.1]{Kunita90}).

\subsection{Stochastic Flows as Random Dynamical Systems} \label{sec:flowtoRDS}

Under quite general assumptions this was done in \citep{Arnold95} and for our purpose in \citep[page 31]{Dimitroff06}.

To construct a random dynamical system in the sense of Section \ref{sec:rds} we need to assume that the semimartingale $F$ has stationary and independent increments, i.e. for all $0 \leq s \leq t$ the $C(\R^d,\R^d)$-valued random variables $F(\cdot, t) - F(\cdot, s)$ and $F(\cdot, t-s)$ have the same distribution and for all $n \geq 0$ and $0 \leq t_1 < t_2 < \dots < t_n$ the random variables $F(\cdot, t_1), F(\cdot, t_2) - F(\cdot, t_1), \dots,F(\cdot, t_{n}) - F(\cdot, t_{n-1})$ are independent. Since the flow $\flow$ satisfies \rref{eq:sde} stationarity and independence is directly transferred to $\flow$. Then we can construct a random dynamical system from a stochastic flow as follows: As in the proof of \citep[Proposition 2.2.1]{Dimitroff06} we can construct the flow $\flow$ on its canonical pathspace $(\tilde \Omega, \tilde{\mathcal{F}}, \tilde{\mathbf{P}})$, where
\begin{align*}
\tilde \Omega &:= C_0\left(\R, C\left(\R^d,\R^d\right)\right) := \left\{f: \R \to C\left(\R^d,\R^d\right): f \text{ is continuous and } f(0) = 0\right\}
\end{align*}
equipped with the topology of uniform convergence on compacts and
\begin{align*}
 \tilde{\mathcal{F}} := \mathcal{B}\left(C_0\left(\R, C\left(\R^d,\R^d\right)\right)\right)
\end{align*}
the Borel $\sigma$-algebra on $\tilde \Omega$. The measure $\tilde{\mathbf{P}}$ on $(\tilde\Omega,\tilde{\mathcal{F}})$ is then defined by $\tilde{\mathbf{P}}(\tilde\omega(0) = \id_{\R^d}) = 1$ and its increments, i.e. for all $n \geq 0$, $0\leq t_1 < t_2 < \dots < t_n$ and all $B \in \mathcal{B}\left(C\left(\R^d,\R^d\right)\right)^{\otimes n}$ set
\begin{align*}
&\tilde{\mathbf{P}}\left(\left(\tilde \omega(t_1), \tilde \omega(t_2)\circ \tilde \omega(t_1)^{-1}, \dots, \tilde \omega(t_{n})\circ \tilde \omega(t_{n-1})^{-1}\right) \in B\right) \\
&\hspace{35ex}= \mathbf{P}\left(\left(\flow_{0,t_1},\flow_{t_1,t_2},\dots,\flow_{t_{n-1},t_n} \right)\in B \right).
\end{align*}
If we now discretize the flow uniformly with step size $1$ then we can define by the stationarity and independence of the flow the measure
\begin{align*}
\nu := \mathbf{P}\circ \flow^{-1}_{0,1}
\end{align*}
on $(\Omega,\mathcal{B}(\Omega))$ (in the sense of Section \ref{sec:rds}) and one can easily see that we are exactly in the situation of Section \ref{sec:rds} with $f_0(\omega) = \tilde\omega(1) = \flow_{0,1}(\bar\omega, \cdot)$.

We will call a probability measure $\mu$ on $\R^{d}$ an {\it invariant} measure of the stochastic flow $\flow$, if it is an invariant measure for the one-point motion of the flow in the sense of discrete (one-step) Markov chains, i.e. for any Borel set $A$ of $\R^{d}$
\begin{align*}
\int_{\bar\Omega} \mu\left(\flow^{-1}_{0,1}(A)\right) \dx \mathbf{P} = \mu(A).
\end{align*}
This definition coincides directly with the definition of invariant measures for random dynamical systems (see Definition \ref{def:invariantMeasure}) via the construction above.

Let us remark that we can use the one-step discretization without loss of generality for our purposes. If we denote $\nu_t := \mathbf{P}\circ \flow^{-1}_{0,t}$ then \citep[Corollary 3.3]{bargen10} implies that for every $t \geq 0$ the entropy satisfies
\begin{align*}\
h_\mu(\mathcal{X}^+(\R^d,\nu_t)) = t h_\mu(\mathcal{X}^+(\R^d,\nu)).
\end{align*}
Thus we will consider the random dynamical system constructed from the one-step discretization of the stochastic flow $\flow$.

\subsection{Pesin's Formula for Stochastic Flows}

Then we can state the main theorem, which says that under some mild regularity assumptions on the driving semimartingale field and a mild integrability assumption on the invariant probability measure the assumptions from Section \ref{sec:rds} are satisfied and hence Pesin's formula holds.

\begin{theorem}
Let $\flow$ be an stochastic flow with driving semimartingale field $F\in B^{k,1}_{ub}$ for some $k \geq 2$ which has stationary and independent increments and let the semimartingale field of the backward flow be also an element of $B^{k,1}_{ub}$. Assume further that $\flow$ has an invariant probability measure $\mu$ which satisfies
\begin{align} \label{eq:assumptionOnMu}
\int_{\R^{d}} \left(\log(\abs{x} + 1)\right)^{1/2} \dx \mu(x) < +\infty.
\end{align}
Then the discretized $\flow$ is a random dynamical system in the sense of Section \ref{sec:rds} and it satisfies Assumptions \ref{ass1} - \ref{ass3}.
\end{theorem}

\begin{proof}
From Section \ref{sec:flowtoRDS} we know that the discretized flow can be seen as a random dynamical system in the sense of Section \ref{sec:rds}. 

Let us prove that the integrability assumptions are satisfied. Since the norm of the derivative of order $k$ can be bounded (in both directions) by the sum of norms of partial derivatives up to order $k$ (neglecting a constant) it suffices to estimate each partial derivative. We will apply \citep[Theorem 2.2]{Imkeller99} to prove that the assumptions from Section \ref{sec:rds} are satisfied.  So let $\alpha$ be a multi index with $\abs{\alpha} = 1$. Since the generating semimartingale field is an element of $B^{k,1}_{ub}$ for $k \geq 2$ by \citep[Theorem 2.2]{Imkeller99} there exists $c, \gamma > 0$ such that the random variable
\begin{align*}
Y_\alpha = \sup_{y \in \R^d} \sup_{0\leq s,t \leq 1} \abs{D^\alpha_y \flow_{s,t}} e^{-\gamma(\log^+\abs{y})^{1/2}} 
\end{align*}
is $\Phi_c$-integrable, where $\Phi_c(x):= \int_1^\infty \exp(-ct^2)x^t \dx t$. By \citep[Lemma 1.1]{Imkeller99} we have for $x \geq 1$ the inequality
\begin{align*}
e^{(\log x)^{2}/4c} e^{-(\log K)^{2}/4c} \leq \Phi_{c}(x),
\end{align*}
where the constant $K$ only depends on $c$ and is defined in \citep[Lemma 1.1]{Imkeller99}. Hence using the inequality $x \leq e^{x^2}$ and the fact that $\Phi_{c}(x) \geq 0$ for $x \geq 0$ we get for each $(\omega,x) \in \Omega^{\N}\times\R^{d}$
\begin{align} \label{eq:ass1estimate}
&\log^{+}\abs{D^{\alpha}_{x}f_{0}(\omega)} = \log^{+}\abs{D^{\alpha}_{x}\flow_{0,1}(\bar\omega)}
\leq \log^{+}\left(Y_{\alpha}\right) + \gamma(\log^+\abs{x})^{1/2}\notag\\
&\hspace{10ex}\leq \mathbf{1}_{\{Y_{\alpha} < 1\}}\Phi_c(Y_\alpha) + \mathbf{1}_{\{Y_{\alpha} \geq 1\}} 2\sqrt{c} \exp\left(\frac{(\log K)^2}{4c}\right) \Phi_c(Y_\alpha)+ \gamma(\log^+\abs{x})^{1/2}
\end{align}
which yields the Assumption \ref{ass1} since the first and second term are integrable with respect to $\mathbf{P}$ where as the third one is integrable with respect to $\mu$ by \rref{eq:assumptionOnMu}. Because of 
\begin{align} \label{eq:LogRelation}
\abs{\log\abs{D_{f_{0}(\omega)x}f_{0}(\omega)^{-1}}} \leq \log^{+}\abs{D_{f_{0}(\omega)x}f_{0}(\omega)^{-1}} + \log^{+}\abs{D_{x}f_{0}(\omega)}
\end{align}
and since the flow property implies $f_{0}(\omega)^{-1}=\flow_{0,1}^{-1} = \flow_{1,0}$ Assumption \ref{ass1b} follows from Assumption \ref{ass1} and from \rref{eq:ass1estimate} applied to the inverse using the invariance of $\mu$.

Assumption \ref{ass2} follows similarly. Let $\abs{\alpha} \leq 2$. Since the exponential map on $\R^{d}$ is a simple translation we have for each $(\omega,x) \in \Omega^{\N}\times\R^{d}$
\begin{align*}
\abs{D^\alpha_\xi F_{(\omega,x),0}} = \abs{D^\alpha_{\exp_x(\xi)} f_0(\omega)}.
\end{align*}
This implies for $(\omega,x) \in \Omega^{\N}\times\R^{d}$
\begin{align} \label{eq:ass2estimate}
&\log^{+}\left(\sup_{\xi \in B_x(0,1)} \abs{D^\alpha_\xi F_{(\omega,x),0}}\right) = \log^{+}\left(\sup_{\xi \in B_x(0,1)} \abs{D^\alpha_{\exp_x(\xi)} f_0(\omega)}\right) \notag\\
&\hspace{6ex}\leq\log^{+}\left(\sup_{\xi \in B_x(0,1)} \abs{D^\alpha_{\exp_{x}(\xi)} \varphi_{0,1}} e^{-\gamma \left(\log^+\abs{\exp_{x}(\xi)}\right)^{1/2}}\right) + \sup_{\xi \in B_x(0,1)} \gamma \left(\log^+\abs{\exp_{x}(\xi)}\right)^{1/2}\notag\\
&\hspace{6ex}\leq\log^{+}\left(Y_{\alpha}\right) + \gamma\left(\log^+(\abs{x} + 1)\right)^{1/2},
\end{align}
which proves via \rref{eq:ass1estimate} the integrability of the positive part and analogously of
\begin{align*}
&\log^{+}\left(\sup_{\xi \in B_x(0,1)} \abs{D^\alpha_{F_{(\omega,x),0}(\xi)} F_{(\omega,x),0}^{-1}}\right).
\end{align*}
Thus Assumption \ref{ass2} follows via \rref{eq:LogRelation}.

Because the determinante of a matrix can be bounded by the Euclidean norm, i.e.
\begin{align*}
\abs{\determinante D_{x}f_{0}(\omega)} \leq {\abs{D_{x}f_{0}(\omega)}}^{d},
\end{align*}
inequality \rref{eq:LogRelation} implies
\begin{align*}
\abs{\log\abs{\determinante D_{x}f_{0}(\omega)}} \leq d \abs{\log{\abs{D_{x}f_{0}(\omega)}}} \leq d \log^{+}{\abs{D_{x}f_{0}(\omega)}} + d \log^{+}{\abs{D_{f_{0}(\omega)x}f_{0}(\omega)^{-1}}},
\end{align*}
which proves Assumption \ref{ass2b} via Assumption \ref{ass1} and \ref{ass1b}.

Finally let us define for $\alpha \leq k$ and $n \in \N$
\begin{align*}
Y_{\alpha}^{n} := \sup_{y \in \R^d} \sup_{0\leq s,t \leq n} \abs{D^\alpha_y \flow_{s,t}} e^{-\gamma(\log^+\abs{y})^{1/2}}.
\end{align*}
Then for fixed $n \in \N$ by \citep[Theorem 2.2]{Imkeller99} there exist $c_{n}, \gamma_{n} > 0$ such that $Y_{\alpha}^{n}$ is $\Psi_{c_{n}}$-integrable and thus Assumption \ref{ass3} follows analogously via \rref{eq:ass2estimate}.
\end{proof}

%% file: pesin.bbl
\providecommand{\bysame}{\leavevmode\hbox to3em{\hrulefill}\thinspace}
\providecommand{\MR}{\relax\ifhmode\unskip\space\fi MR }
\providecommand{\MRhref}[2]{%
  \href{http://www.ams.org/mathscinet-getitem?mr=#1}{#2}
}
\providecommand{\href}[2]{#2}
\begin{thebibliography}{10}

\bibitem{Arnold98}
L.~Arnold, \emph{Random dynamical systems}, Springer Monographs in Mathematics,
  Springer-Verlag, Berlin, 1998.

\bibitem{Arnold95}
L.~Arnold and M.~Scheutzow, \emph{Perfect cocycles through stochastic
  differential equations}, Probab. Theory Related Fields \textbf{101} (1995),
  no.~1, 65--88.

\bibitem{bahnmueller95}
J.~Bahnm{\"u}ller and T.~Bogensch{\"u}tz, \emph{A {M}argulis-{R}uelle
  inequality for random dynamical systems}, Arch. Math. (Basel) \textbf{64}
  (1995), no.~3, 246--253.

\bibitem{Barreira07}
L.~Barreira and Ja.~B. Pesin, \emph{Nonuniform hyperbolicity}, Encyclopedia of
  Mathematics and its Applications, vol. 115, Cambridge University Press,
  Cambridge, 2007.

\bibitem{Biskamp11b}
M.~Biskamp, \emph{Absolute continuity theorem for random dynamical systems on
  {$\R^d$}}, preprint, 2011.

\bibitem{Dimitroff06}
G.~Dimitroff, \emph{Some properties of isotropic {B}rownian and
  {O}rnstein-{U}hlenbeck flows}, Ph.D. thesis, Technische Universit\"at Berlin,
  2006, \url{http://opus.kobv.de/tuberlin/volltexte/2006/1252/}.

\bibitem{Fathi83}
A.~Fathi, M.-R. Herman, and J.-C. Yoccoz, \emph{A proof of {P}esin's stable
  manifold theorem}, Geometric dynamics ({R}io de {J}aneiro, 1981), Lecture
  Notes in Mathematics, vol. 1007, Springer-Verlag, Berlin, 1983, pp.~177--215.

\bibitem{Imkeller99}
P.~Imkeller and M.~Scheutzow, \emph{On the spatial asymptotic behavior of
  stochastic flows in {E}uclidean space}, Ann. Probab. \textbf{27} (1999),
  no.~1, 109--129.

\bibitem{Katok86}
A.~Katok, J.-M. Strelcyn, F.~Ledrappier, and F.~Przytycki, \emph{Invariant
  manifolds, entropy and billiards; smooth maps with singularities}, Lecture
  Notes in Mathematics, vol. 1222, Springer-Verlag, Berlin, 1986.

\bibitem{Kifer86}
Y.~Kifer, \emph{Ergodic theory of random transformations}, Progress in
  Probability and Statistics, vol.~10, Birkh\"auser Boston Inc., Boston, MA,
  1986.

\bibitem{Kunita90}
H.~Kunita, \emph{Stochastic flows and stochastic differential equations},
  Cambridge Studies in Advanced Mathematics, vol.~24, Cambridge University
  Press, Cambridge, 1990.

\bibitem{Ledrappier88}
F.~Ledrappier and L.-S. Young, \emph{Entropy formula for random
  transformations}, Probab. Theory Related Fields \textbf{80} (1988), no.~2,
  217--240.

\bibitem{Liu95}
P.-D. Liu and M.~Qian, \emph{Smooth ergodic theory of random dynamical
  systems}, Lecture Notes in Mathematics, vol. 1606, Springer-Verlag, Berlin,
  1995.

\bibitem{Pesin76}
Ja.~B. Pesin, \emph{Families of invariant manifolds that correspond to nonzero
  characteristic exponents}, Izv. Akad. Nauk SSSR Ser. Mat. \textbf{40} (1976),
  no.~6, 1332--1379, 1440.

\bibitem{Pesin77}
\bysame, \emph{Characteristic {L}japunov exponents, and smooth ergodic theory},
  Uspehi Mat. Nauk \textbf{32} (1977), no.~4 (196), 55--112, 287.

\bibitem{Pesin77b}
\bysame, \emph{A description of the {$\pi $}-partition of a diffeomorphism with
  an invariant measure}, Mat. Zametki \textbf{22} (1977), no.~1, 29--44.

\bibitem{Rohlin67}
V.~A. Rohlin, \emph{Lectures on the entropy theory of transformations with
  invariant measure}, Uspehi Mat. Nauk \textbf{22} (1967), no.~5 (137), 3--56.

\bibitem{Ruelle79}
D.~Ruelle, \emph{Ergodic theory of differentiable dynamical systems}, Inst.
  Hautes \'Etudes Sci. Publ. Math. (1979), no.~50, 27--58.

\bibitem{bargen10}
H.~van Bargen, \emph{Ruelle's inequality for isotropic {O}rnstein-{U}hlenbeck
  flows}, Stoch. Dyn. \textbf{10} (2010), no.~1, 143--154.

\end{thebibliography}
